\documentclass[12pt]{article}
\usepackage{amsmath,amsfonts,amssymb,amsthm}
\usepackage[mathscr]{eucal}

\newcommand{\w}{\omega}

\newcommand{\id}{{\rm id}}
\newcommand{\Hom}{{\rm Hom\,}}

\newcommand{\End}{{\rm End}}

\newcommand{\Res}{{\rm Res\,}}

\DeclareMathOperator{\vac}{{\rm Vac\,}}

\newcommand{\fusion}[3]{{\binom{#3}{#1\;#2}}}

\newcommand\Z{\mathbb{Z}}
\newcommand\Zpos{\Z_{\geq0}}
\newcommand\Zplus{\Z_{>0}}

\newcommand\C{\mathbb{C}}
\newcommand\N{\mathbb{N}}
\newcommand\h{\mathfrak{h}}

\newcommand{\NO}{\,{\raise0.25em\hbox{$\mathop{\hphantom {\cdot}}\limits^{_{\circ}}_{^{\circ}}$}}\,}

\providecommand{\abs}[1]{\lvert#1\rvert}

\newcommand{\Fusion}{{\rm F}}

\renewcommand{\vac}{|0\rangle}
\newcommand{\vact}{|\theta\rangle}

\newcommand{\frakg}{\mathfrak{g}}
\newcommand{\frakh}{\mathfrak{h}}

\newcommand{\ket}[1]{|#1\rangle}

\newtheorem{theorem}{Theorem}[section]
\newtheorem{proposition}[theorem]{Proposition}
\newtheorem{lemma}[theorem]{Lemma}

\theoremstyle{definition}
\newtheorem{definition}[theorem]{Definition}

\theoremstyle{remark}
\newtheorem{remark}[theorem]{\bf Remark}

\numberwithin{equation}{section}

\begin{document}
\begin{large}
\begin{center}
Intertwining operators and fusion rules for vertex operator algebras\\
arising from  symplectic fermions
\end{center}
\end{large}

\begin{center}
Toshiyuki Abe${}^1$
, Yusuke Arike${}^2$
\\
\bigskip
${}^1$ Graduate School of Science and Engineering,\\
Ehime University\\
2-5, Bunkyocho, Matsuyama, Ehime 790-8577, Japan\\
abe@ehime-u.ac.jp

\bigskip 
${}^2$ Department of Pure and Applied Mathematics\\
Graduate School of Information Science and Technology\\
Osaka University\\
y-arike@cr.math.sci.osaka-u.ac.jp
\end{center}
\vskip2ex
\abstract{
We determine fusion rules (dimensions of the space of 
intertwining operators) among simple modules for the vertex operator algebra obtained as
an even part of the symplectic fermionic vertex operator superalgebra.
By using these fusion rules we show that the fusion algebra of this vertex operator algebra
is isomorphic to the group algebra of the Klein four group over $\Z$.
}

\tableofcontents
\section{Introduction}\label{sect-intro}
The aim of this paper is to determine fusion rules among simple modules
for the vertex operator algebra $\mathscr{F}^+$ arising from symplectic fermions.

A fusion rule is defined to be the dimension of the space of intertwining operators among three modules
for a vertex operator algebra.
It is known that intertwining operators give rise to $3$-point functions over the  projective line (\cite{Z1}).

The vertex operator algebra $\mathscr{F}^+$ is an even part of the symplectic fermionic vertex operator superalgebra
which is constructed from a $2d$-dimensional symplectic vector space (\cite{Abe07}).
It is shown in \cite{Abe07} that $\mathscr{F}^+$ satisfies Zhu's finiteness condition
but is not rational and that  the number of simple modules is four, which is independent of $d$.

In order to determine fusion rules among simple modules, we first restrict our attention to the case $d=1$
(we denote $\mathscr{F}^+$ for $d=1$ by $\mathscr{T}^+$).
In this case we determine fusion rules
by using the notion of Frenkel--Zhu bimodules (\cite{FZ}). 
Frenkel--Zhu bimodules are quotient spaces of modules for a vertex operator algebra
together with the left and right actions of Zhu's algebra. 
However, we should be careful when we apply the Frenkel--Zhu theory 
to our vertex operator algebra since \cite{FZ} is not precisely correct.
As it is known in \cite{Li},
we only use Frenkel--Zhu bimodules 
to get the upper bounds of fusion rules.
For this purpose we determine generators of Frenkel--Zhu bimodules of simple $\mathscr{T}^+$-modules.
We can show that Frenkel--Zhu bimodules are generated by the images of the lowest weight spaces of simple modules.
The lower bounds of fusion rules are determined by giving nontrivial intertwining operators.

The fusion algebra of $\mathscr{T}^+$ is a free $\Z$-module generated by all inequivalent simple $\mathscr{T}^+$-modules
with the multiplication defined by
\begin{equation*}
M\times N=\sum_{L:\text{ simple}}\dim_\C I_{\mathscr{T}^+}\fusion{M}{N}{L} L
\end{equation*}
where $\dim_\C I_{\mathscr{T}^+}\fusion{M}{N}{L}$ is a fusion rule.
Fusion rules for simple $\mathscr{T}^+$-modules yield that
the fusion algebra of  $\mathscr{T}^+$ is isomorphic to the group algebra of the Klein four group 
$\Z/2\Z\times\Z/2\Z$ over $\Z$.

For the case $d>1$ we use the fact that $(\mathscr{T}^+)^{\otimes d}$ is a vertex operator subalgebra
of $\mathscr{F}^+$ but does not share the Virasoro element with $\mathscr{F}^+$.
Then it follows that every simple $\mathscr{F}^+$-module decomposes into a direct sum
of tensor product of simple $\mathscr{T}^+$-modules as a $(\mathscr{T}^+)^{\otimes d}$-module.
Then by the results in \cite{ADL}, we obtain our main result: the fusion algebra for $d>1$ is isomorphic to the one for $d=1$.

In \cite{Abe07}, the first author construct indecomposable reducible $\mathscr{F}^+$-modules on which $L_0$ do not act semisimply.
It is pointed out in \cite{Mil,Ar} that intertwining operators among such indecomposable modules
involve logarithmic terms. 
This issue will be studied in the forthcoming papers.

The organization of this paper is described as follows.
In section 2 we recall basics about vertex operator algebras, intertwining operators 
and Frenkel--Zhu bimodules.
In section 3 we introduce the symplectic fermionic vertex operator superalgebra 
$\mathscr{F}$ and recall simple $\mathscr{F}^+$-modules. 
In section 4 we determine generators of Frenkel--Zhu bimodules of simple $\mathscr{F}^+$-modules
for the case $d=1$. 
Final section is devoted to determine fusion rules among simple modules for $d=1$ and then $d>1$.
The appendix is devoted to derive a formula of a determinant used in section 4.
\vskip 1ex
\noindent
{\bf Acknowledgment}
The first author is partly supported by Grant in Aid for Young Scientist (B) 23740022
and the second author is partly supported by Grant in Aid for Young Scientist (B) 23740019.

\section{Preliminaries}\label{sect-2}
\subsection{Vertex operator algebras and their modules}\label{subsect-2-1}
\begin{definition}\label{def:chiral}
A quartet $(V,\,Y,\,\omega,\,\vac)$ is called 
a \textit{vertex operator algebra} if it satisfies the following conditions:
\vskip1ex
\noindent
(1) The vector space $V$ is  $\N$-graded  
$V=\bigoplus_{n=0}^\infty V_n$ such that $\dim_\C V_n<\infty$ 
for all nonnegative integers $n$. Any element $a$ of 
$V_n$ is called a \textit{homogeneous} element of weight $n$; 
we denote $|a|=n$ for any $v\in V_n$.
\vskip1ex
\noindent
(2) The elements $\vac\in V_0$ and $\omega\in V_2$ are called 
the \textit{vacuum} and the \textit{Virasoro element}, respectively.
\vskip1ex
\noindent
(3) There is a linear map 
$Y(-,z):V\rightarrow\End\,V[[z,z^{-1}]]\;(v\mapsto  
Y(a,z)=\sum_{n\in\Z}a_{(n)}z^{-n-1})$ 
satisfying 
$a_{(n)}V_m\subset V_{m+\abs{a}-n-1}$ for all nonnegative integers $m$ 
such that $\vac_{(-1)}= \id_V,\,\vac_{(n)}=0\,(n\neq -1)$, and
$a_{(-1)}\ket{0}=a,\, a_{(n)}\ket{0}=0\,(n>-1)$ for all $a\in V$. 
The $L_n=\omega_{(n+1)}\,(n\in\Z)$ and $\vac_{(-1)}=\id_V$ give rise to 
a module structure for the Virasoro algebra  on $V$ with central charge 
$c_V\in\C$, and $L_0$ is a grading operator, that is, $L_0v=nv$ 
for any $v\in V_n$.
\vskip1ex
\noindent
(4) $(L_{-1}a)_{(n)}=-na_{(n-1)}$ for all $a\in V$ and integers $n$.
\vskip 1ex
\noindent
(5) Let $a^1,\,a^2\in V$ and 
$m,\,n\in\Z$. Then the commutator formula holds:
\begin{equation*}
[{a^1}_{(m)},{a^2}_{(n)}]=
\sum_{j=0}^{\infty}
\binom{m}{j}(v^1_{(j)}a^2)_{(m+n-j)}.
\end{equation*}
(6) Let $a^1,\,a^2\in V$ and 
$m,\,n\in\Z$. Then the associativity formula holds:
\begin{equation*}
(a^1_{(m)}a^2)_{(n)}=
\sum_{j=0}^\infty(-1)^j\binom{m}{j}
\left(
a^1_{(m-j)}a^2_{(n+j)}
-(-1)^{m}a^2_{(m+n-j)}a^1_{(j)}
\right).
\end{equation*}
\end{definition}

Next we shall recall from \cite{NT1} and \cite{MNT} the notions of modules 
for vertex operator algebras.

Let $V$ be a vertex operator algebra and let $\frakg$ be the associated current Lie algebra.
Recall from \cite{MNT} that $\mathscr{U}(\frakg)$ is the current algebra and $I_d\,(d\in\Z)$ is the degreewise completion of a filtration on
the universal enveloping algebra $U(\frakg)$ of $\frakg$ (see \cite{MNT, NT1}).

A {\it weak $V$-module} is defined to be a $\mathscr{U}(\frakg)$-module and a
{\it $V$-module} is a finitely generated $\mathscr{U}(\frakg)$-module $M$ with following properties:
For any $u\in M$ there exists an integer $d$ such that $I_{d}\,u=0$ and 
for any $u\in M$
the vector space linearly spanned by vectors
\begin{equation*}
a^1_{(\abs{a^1}-1+n_1)}a^2_{(\abs{a^2}-1+n_2)}
\dotsb a^s_{(\abs{a^s}-1+n_s)}u\quad
(\forall a^i\in V,\,n_1+n_2+\dotsb+n_s\geq0)
\end{equation*}
is finite-dimensional

\begin{definition}
Let $V$ be a vertex operator algebra and let $C_2(V)$ 
be the vector subspace of $V$ which is linearly spanned 
by elements of the form $a_{(-2)}b\,(n\geq1)$. If the quotient 
space $V/C_2(V)$ is finite-dimensional we say that $V$ satisfies
\textit{Zhu's finiteness condition}.
\end{definition}

\begin{proposition}[{\cite[Corollary 3.2.8]{NT1}}]
\label{proposition:2.1.8}
Let $V$ be a vertex operator algebra and 
let $M$ be a $V$-module. Then $M=\bigoplus_{r\in\C}M_{(r)}$
where 
$M_{(r)}=\{m\in M\;|\; (L_0-r)^nm=0 \text{ for a positive integer $n$}\}$. 
If $V$ satisfies Zhu's finiteness condition 
then $dim_\C M_{(r)}<\infty$ for any complex number $r$.
\end{proposition}

It is not difficult to show the following proposition.

\begin{proposition}\label{prop-2-1-5}
Let $M$ be a simple $V$-module. 
Then there exists a complex number $r$ such that 
$M=\bigoplus_{n\in\N}M_{r+n},\, M_{r}\neq0$, where 
$M_{r+n}=\{m\in M\;|\; (L_0-r-n) m=0\}$. Moreover,  
$\dim_\C M_{r+n}<\infty\,(n\in\N)$ if $V$ satisfies Zhu's finiteness condition.
\end{proposition}

\begin{remark}
Let $V$ be a vertex operator algebra satisfying Zhu's finiteness condition
and $M$ be a $V$-module.
Then we can define {the contragredient module} $D(M)$ of $M$.
It is shown that $D(D(M))$ is isomorphic to $M$  (see e.g. \cite[Proposition 4.2.1]{NT1}).
\end{remark}
\subsection{Intertwining operators}\label{subsect-2-2}
Let $V$ be a vertex operator algebra and 
let $L,\,M,\,N$ be weak $V$-modules.
An {\it intertwining operator} of the type $\fusion{L}{M}{N}$ is a linear map
\begin{equation*}
I(-,z):L\to\Hom_\C(M,N)\{z\},\quad 
u\mapsto
I(u,z)=\sum_{\alpha\in\C}u_{(\alpha)}z^{-\alpha-1}
\end{equation*}
satisfying the following conditions:
\medskip

\noindent
(i)
For any $u\in L$, $v\in M$ and a complex number $\alpha$, there exists a non-negative integer $m$ such that
$u_{(\alpha+n)}v=0$ for all $n\ge m$.

\noindent
(ii)
For any $u\in L$, we have
\begin{equation*}
I(L_{-1}u,z)=\frac{d}{dz}I(u,z).
\end{equation*}

\noindent
(iii)
For any $a\in V$, $u_1\in L$, $u_2\in M$, $\alpha\in\C$, and $p,q\in\Z$, we have
\begin{equation*}\label{eq-borcherds}
\begin{split}
&\sum_{i=0}^\infty\binom{p}{i} (a_{(q+i)}u_1)_{(\alpha+p-i)}u_2\\
&=\sum_{i=0}^\infty(-1)^i\binom{q}{i}(a_{(p+q-i)}(u_1)_{(\alpha+i)}u_2-(-1)^q(u_1)_{(\alpha+q-i)}a_{(p+i)}u_2). 
\end{split}
\end{equation*}

We denote the space of all intertwining operators of the type $\fusion{L}{M}{N}$
by $I_V\fusion{L}{M}{N}$ and call $\dim_\C I_V\fusion{L}{M}{N}$ 
the {\it fusion rule} of the type $\fusion{L}{M}{N}$.

Fusion rules have obvious symmetries:

\begin{proposition}[{\cite[Proposition 5.5.2]{FHL}}]\label{prop-1-2}
Let $V$ be a vertex operator algebra satisfying Zhu's finiteness condition and let $L,\,M,\,N$ be
$V$-modules.
Then there exist canonical vector space isomorphisms
\begin{equation*}
I_V\fusion{L}{M}{N}\cong I_V\fusion{M}{L}{N}\cong I_V\fusion{L}{D(N)}{D(M)}.
\end{equation*}
\end{proposition}

The following propositions and the theorem play important roles in the paper.

\begin{proposition}[{cf. \cite[Proposition 7.4]{Zhu96}}]\label{prop-1-5}
Let $V$ be a vertex operator algebra and let $M,\,N$ be simple $V$-modules.
Then 
\begin{equation*}
\dim_\C I_V\fusion{V}{M}{N}=
\begin{cases}
1\quad&M\cong N,\\
0\quad&M\not\cong N.
\end{cases}
\end{equation*}
\end{proposition}

\begin{proposition}[{\cite[Proposition 2.9]{ADL}}]\label{prop-1-3}
Let $V$ be a vertex operator algebra and let $L,\,M,\,N$ be 
$V$-modules. Suppose that $L$ and $M$ are simple $V$-modules.
Let $U$ be a vertex operator subalgebra of $V$ with the same Virasoro element
and let  $L^\prime$ and $M^\prime$ be simple $U$-submodules of $L$ and $M$, respectively.
Then we have
\begin{equation*}
\dim_\C I_V\fusion{L}{M}{N}\le\dim_\C I_U\fusion{L^\prime}{M^\prime}{N}.
\end{equation*}
\end{proposition}

\begin{theorem}[{\cite[Theorem 2.10]{ADL}}]\label{thm-1-4}
Let $V^1$ and $V^2$ be vertex operator algebras
and let $L^1,\,M^1,\,N^1$ be $V^1$-modules and let
$L^2,\,M^2,\,N^2$ be $V^2$-modules.
Suppose that 
\begin{equation*}
\dim_\C I_{V^1}\fusion{L^1}{M^1}{N^1}<\infty\text{ or }\dim_\C I_{V^2}\fusion{L^2}{M^2}{N^2}<\infty.
\end{equation*}
Then there exists an isomorphism
\begin{equation*}
I_{V^1}\fusion{L^1}{M^1}{N^1}\otimes I_{V^2}\fusion{L^2}{M^2}{N^2}\cong 
I_{V^1\otimes V^2}\fusion{L^1\otimes L^2}{M^1\otimes M^2}{N^1\otimes N^2}.
\end{equation*}
\end{theorem}

\subsection{Bimodules of Frenkel-Zhu}\label{subsect-2-3}
In this subsection we recall the notion of Frenkel-Zhu bimodules.

Let $V$ be a vertex operator algebra and let $M$ be a $V$-module.
Let $O(M)$ be the vector subspace linearly spanned by
\begin{equation*}
a\circ u=\sum_{i=0}^\infty\binom{|a|}{i}a_{(i-2)}u
\end{equation*}
for all homogeneous $a\in V$ and $u\in M$.
We set $A(M)=M/O(M)$ and define the bilinear operation 
$*:V\times M\rightarrow M$
\begin{equation*}
a\ast u=\sum_{i=0}^\infty\binom{|a|}{i}a_{(i-1)}u
\end{equation*}
for any $a\in V$ and $u\in M$.

It is shown in \cite{Zhu96} that the bilinear operation $\ast$ induces 
a structure of an associative algebra, called {\it Zhu's algebra}, on
$A(V)$ with the unity $[\vac]=\vac+O(V)$ and a central element 
$[\omega]=\omega+O(V)$.

Let $\Omega(M)$ be the vector space consisting of elements in 
$M$ which satisfy
\begin{equation*}
a_{(|a|-1+n)}u=0
\end{equation*}
for all $a\in V$ and positive integers $n$.
Then $\Omega(M)$ is canonically a left $A(V)$-module with the action
$o(a)u:=a_{(|a|-1)}u$.

\begin{theorem}[{\cite[Theorem 2.2.2]{Zhu96}}]
The functor $\Omega$ gives rise to a bijection 
between the complete set of inequivalent simple $V$-modules 
and that of inequivalent simple $A(V)$-modules.
\end{theorem}

It is shown in \cite[Theorem 1.5.1]{FZ} that
$A(M)$ is an $A(V)$-bimodule with the above left operation $\ast$ 
and the right operation given by
\begin{equation*}
u\ast a=\sum_{i=0}^\infty\binom{|a|-1}{i}a_{(i-1)}u.
\end{equation*}

Then we have:
\begin{lemma}[{\cite[Lemma 3.5]{Abe1}}]\label{lem-1-10}
Let $M$ be a weak $V$-module.
Then
\begin{equation*}
L_{-1}u\equiv \omega*u-u*\omega-|u|u\mod O(M)
\end{equation*}
for any $L_0$-eigenvector $u\in M$ of eigenvalue $|u|$.
\end{lemma}

It is well known that a simple $V$-module $M$ decomposes into 
$L_0$-eigenspaces $M=\bigoplus_{n=0}^\infty M_{r+n},\,M_{r}\neq\{0\}$ 
with a complex number $r$. 
This complex number $r$ is called a {\it conformal weight}\,
of the simple module $M$. 

The importance of $A(V)$-bimodules for determining fusion rules
is explained the following theorem.

\begin{theorem}[{\cite[Theorem 1.5.2]{FZ}, \cite[Proposition 2.10]{Li}}]\label{thm-1-1}
Let $V$ be a vertex operator algebra. Let $L,\,M$ and $N$
be simple $V$-modules with conformal weights $r_1,\,r_2$ and $r_3$, respectively. Then there is an injective linear map 
\begin{equation}\label{eq-contraction}
I_V\fusion{L}{M}{N}\to 
(\Omega(N))^*\otimes_{A(V)}A(L)\otimes_{A(V)}
\otimes_{A(V)}\Omega(M),
\end{equation}
where $(\Omega(N))^*=\Hom_\C(\Omega(N),\C)$ is viewed 
as a right $A(V)$-module in a natural way.
\end{theorem}

We call  the right-hand side of \eqref{eq-contraction} the {\it contraction} of $(\Omega(N))^*,\,A(L)$ and $\Omega(M)$
and hereafter we denote it by $(\Omega(N))^*\cdot A(L)\cdot\Omega(M)$. 
As an immediate consequence we have
\begin{equation*}
\dim_\C I_V\fusion{L}{M}{N}\leq
\dim_\C (\Omega(N))^*\cdot A(L)\cdot\Omega(M).
\end{equation*}

\section{The vertex operator algebra $\mathscr{F}^+$ and their simple modules}\label{sect-3}
\subsection{The symplectic fermionic vertex operator superalgebra $\mathscr{F}$}\label{subsect-3-1}
Let $\h$ be a $2d$-dimensional vector space over $\C$ 
with a non-degenerate skew-symmetric bilinear form 
$\langle\,\,,\,\rangle$.
Let $\{e^i,f^i\;|\;i=1,\,\dotsc,\,d\,\}$ be a basis of $\h$ such that
\begin{equation*}
\langle e^i,e^j\rangle=\langle f^i,f^j\rangle=0\text{ and }
\langle e^i,f^j\rangle=-\delta_{ij}
\end{equation*}
for all $1\leq i,\,j\leq d$. There is a Lie superalgebra structure on 
the affinization on $\h$, i.e., $\hat\h=\h\otimes\C[t,t^{-1}]\oplus\C K$ 
with even part $\C K$, odd part $\h\otimes\C[t,t^{-1}]$.
The supercommutation relations 
\begin{align*}
&[\phi\otimes t^{m},\psi\otimes t^{n}]_{+}
=m\langle\phi,\psi\rangle\delta_{m+n,0}K\quad(\phi,\psi\in\h,m,n\in\Z),\\
&[K,\hat\h]=0.
\end{align*}

Let  $\C_+$ be a one-dimensional module  of the Lie supersubalgebra 
$\hat\h^+=\h\otimes\C[t]\oplus\C K$ of $\hat\h$ on which 
$\h\otimes\C[t]$ acts trivially and $K$ acts as $1$. 
We denote by $\mathscr{F}$ the induced module 
${\rm Ind}^{\hat\h}_{\hat\h^+}\C_+$.   
It is immediate that $\mathscr{F}$ is isomorphic to the exterior algebra 
$\bigwedge(\hat\h_-)$ as vector spaces 
where $\hat\h_-=\h\otimes t^{-1}\C[t^{-1}]$ .

We denote by $h_{(n)}$ the action of $h\otimes t^{n}$ on 
$\mathscr{F}$ for any $h\in\h$ and $n\in\Z$. 
Then any vector in $\mathscr{F}$ is a linear combination 
of the vector $|0\rangle=1\otimes1$ and the vectors of the form 
\begin{equation}\label{aiueoh}
v=h^{1}_{(-n_1)}\dotsb h^{r}_{(-n_r)}|0\rangle\quad 
(h^1,\,\dotsc,\,h^{r}\in\h,\,n_1,\dotsc,\,n_r\in\Zplus).
\end{equation}

Let $\mathscr{F}^+$ and $\mathscr{F}^-$ be the subspaces of 
$\mathscr{F}$ linearly spanned by vectors \eqref{aiueoh} 
for even $r$ 
and odd $r$, respectively. As shown in \cite{Abe07} the $\hat{\h}$-module $\mathscr{F}$ becomes  a  simple vertex operator superalgebra 
with the even part $\mathscr{F}^+$ and the odd part $\mathscr{F}^-$. 
The vertex operator $Y(v,z)$ associated with a vector $v$ of the form 
\eqref{aiueoh} is  
\begin{equation*}
Y(v,z)=\NO \partial^{(n_1-1)}h^{1}(z)\dotsb\partial^{(n_r-1)}h^{r}(z)\NO
\end{equation*}
where
\begin{equation*}
h(z)=\sum_{n\in\Z}h_{(n)}z^{-n-1}\quad (h\in \h),\quad
\partial^{(k)}=\dfrac{1}{k!}\left(\dfrac{d}{dz}\right)^{k},
\end{equation*}
and the normal ordering procedure $\NO-\NO$ is defined as 
\begin{equation*}
\NO h^{1}_{(m_1)}\cdots h^{r}_{(m_r)}\NO=
\begin{cases}
h^{1}_{(m_1)}\NO h^{2}_{(m_2)}\dotsb h^{r}_{(m_r)}\NO
&\text{if $m_1< 0$},\\
&\\
(-1)^{r-1}\NO h^{2}_{(m_2)}\dotsb h^{r}_{(m_r)}\NO h^{1}_{(m_1)}
&\text{if $m_1\geq 0$},
\end{cases}
\end{equation*}
recursively for all $h^i\in\h$ and $m_i\in\Z$.  
The vacuum vector is $|0\rangle$ and 
\begin{equation*}
\w=\sum_{i=1}^{d}e^i_{(-1)}f^i_{(-1)}\vac
\end{equation*}
is a conformal vector of central charge $-2d$.
Set $Y(\omega,z)=\sum_{n\in\Z}L_nz^{-n-2}$.
Then $[L_0, h_{(-n)}]=nh_{(-n)}$ for all $h\in\frakh$ and integers $n$ (see \cite{Abe07}).
It is not difficult to see that an element of the form  
\eqref{aiueoh} has weight $n_1+\cdots+n_r$.
Hence $\mathscr{F}_0=\C|0\rangle$ and 
$\mathscr{F}_1=\langle\; h_{(-1)}|0\rangle\;|\;\h\in \h\;\rangle_{\C}\cong \h$ 
so that we can identify $\h$ with $\mathscr{F}_1$ 
by the  correspondence 
$h\mapsto h_{(-1)}|0\rangle$ for any $h\in\mathfrak{h}$.
We call the vertex operator superalgebra $\mathscr{F}$ 
the \textit{symplectic fermionic  vertex operator superalgebra}. 

\subsection{Twisted $\mathscr{F}$-modules}\label{subsect-3-2}
The vertex operator superalgebra $\mathscr{F}$ has the automorphism 
$\theta$ defined by $\theta(u)=\pm u$ for $u\in\mathscr{F}^{\pm}$,
respectively.  In this section we will construct a\, $\theta$-twisted module
of the vertex operator superalgebra $\mathscr{F}$ following \cite{Abe07}. 
This construction is resemble to the one of  twisted modules 
for the Heisenberg vertex operator algebra. 

Let $\hat\h_{\theta}=\h \otimes t^{\frac{1}{2}}\C[t,t^{-1}] \oplus\C K$.
Then $ \hat\h_{\theta}$ is a Lie superalgebra with even part $\C K$ and
odd part $\h \otimes t^{\frac{1}{2}}\C[t,t^{-1}]$. 
Its supercommutation relations
are given by
\begin{align*}
&[h^1\otimes t^m,h^2\otimes t^n]_+
=m\langle h^1,h^2\rangle \delta_{m+n,0}K
\quad\bigl(h^1,\,h^2\in\h,\,m,n\in\frac{1}{2}+\Z\bigr),\\
&[K,\hat\h_\theta]=0.
\end{align*}

Let $\C \vact$ be a one-dimensional module 
for the subalgebra 
$\hat\h_\theta^+=\h\otimes t^{\frac{1}{2}}\C[t]\oplus\C K$ 
of $\hat\h_\theta$ defined by $\hat\h_\theta^+|\theta\rangle=0$ and 
$K|\theta\rangle=|\theta\rangle$. Then we obtain an induced $\hat\h_\theta$-module $\mathscr{F}_t
={\rm Ind}_{\hat\h_\theta^+}^{\hat\h_\theta}\C |\theta\rangle$.
By abuse of a notation, but for short, we denote the action 
of $h\otimes t^n\,(h\in\h,\,n\in\frac{1}{2}+\Z)$ on $\mathscr{F}_t $ by $h_{(n)}$
and introduce a field $h(z)=\sum_{n\in\frac{1}{2}+\Z}h_{(n)}z^{-n-1}$
on $\mathscr{F}_t$.

We define a field $W(v,z)$ associated with a vector  $v$ 
of the type \eqref{aiueoh} by 
\begin{equation*}
W(v,z)=\NO \partial^{(n_1-1)}h^{1}(z)\dotsb\partial^{(n_r-1)}h^{r}(z)\NO,
\end{equation*}
where the normal ordering procedure $\NO-\NO$ is defined as 
\begin{equation*}
\NO h^{1}_{(m_1)}\dotsb h^{r}_{(m_r)}\NO=
\begin{cases}
h^{1}_{(m_1)}\NO h^{2}_{(m_2)}\dotsb h^{r}_{(m_r)}\NO
&\text{if $m_1< 0$},\\
&\\
(-1)^{r-1}\NO h^{2}_{(m_2)}\dotsb h^{r}_{(m_r)}\NO h^{1}_{(m_1)}
&\text{if $m_1> 0$}
\end{cases}
\end{equation*}
for any $h^i\in\h$ and $m_i\in\frac{1}{2}+\Z$, recursively. 

Let $c_{mn}$ be the rational numbers defined by
\begin{equation*}
\sum_{m,n\in\Z_{\geq 0}}c_{mn}x^{m}y^n
=-\log\Bigl(\frac{(1 + x)^{\frac{1}{2}}+(1+y)^{\frac{1}{2}}}{2}\Bigr),
\end{equation*}
and set
\begin{equation*}
\Delta(z)=\sum_{m,\,n\geq 0}\sum_{i=1}^{d}c_{mn}e^i_{(m)}f^i_{(n)}z^{-m-n}.
\end{equation*}

Let define a filed  $Y^\theta(v,z)$ by 
$Y^\theta(v,z)=W(e^{\Delta(z)}v,z)$ for any $v\in \mathscr{F}$. 
Then  the pair $(\mathscr{F}_t ,Y^\theta(-,z))$ 
gives rise to a $\theta$-twisted $\mathscr{F}$-module  
(see \cite{Xu98} for the definition of twisted modules).

The conformal weight of $\mathscr{F}_t$ with respect to the operator $L_0=\omega_{(1)}$ is $-\frac{1}{8}$ and we have 
\begin{equation*}
\mathscr{F}_t 
=\bigoplus_{n=0}^\infty \mathscr{F}_{t,-\frac{1}{8}+\frac{n}{2}}.
\end{equation*}
It is obvious that $\mathscr{F}_{t,-\frac{1}{8}}=\C |\theta\rangle$ and 
$\mathscr{F}_{t,\frac{3}{8}}
=\langle\; h_{(-\frac{1}{2})}|\theta\rangle\;|\;h\in\h\;\rangle_\C$.
Furthermore, we see that $\mathscr{F}_t $ is a simple $\theta$-twisted module. 

\begin{remark}
The even part $\mathscr{F}_t ^+$ and the odd part $\mathscr{F}_t ^-$ of 
$\mathscr{F}_t $ are respectively spanned by elements of the form 
\begin{equation*}
h^{1}_{(-n_1)}\cdots h^{r}_{(-n_r)}|\theta\rangle
\end{equation*}
for even and odd non-negative integers $r$, $h^1,\,\dotsc,\,h^r\in\h$ and $n_1,\dotsc,n_r\in\frac{1}{2}+\Zplus$.  
Obviously we have 
\begin{equation*}
\mathscr{F}_t ^+=\bigoplus_{n=0}^\infty \mathscr{F}_{t,-\frac{1}{8}+n},\quad 
\mathscr{F}_t ^-=\bigoplus_{n=0}^\infty \mathscr{F}_{t,n+\frac{3}{8}}. 
\end{equation*}
\end{remark}

\subsection{Simple $\mathscr{F}^+$-modules}\label{subsect-3-3}
It is obvious that the even part $\mathscr{F}^+$ of the symplectic fermionic vertex operator superalgebra $\mathscr{F}$
is a vertex operator algebra of central charge $-2d$
and that the odd part $\mathscr{F}^-$ is an $\mathscr{F}^+$-module.
Since $\mathscr{F}^+$ is spanned by vectors with even length,
the twisted vertex operator $Y^\theta(v,z)$ for $v\in\mathscr{F}^+$ belongs to $\End_\C(\mathscr{F}_t)[[z,z^{-1}]]$,
which shows that $\mathscr{F}^\pm_t$ are $\mathscr{F}^+$-modules.

\begin{theorem}[{\cite[Theorem 3.13, Theorem 4.2]{Abe07}}]\label{them1}
\begin{enumerate}
\item[{\rm(i)}]
The even part $\mathscr{F}^+$ is a simple vertex operator algebra satisfying Zhu's finiteness condition..
\item[{\rm (ii)}]
Any simple $\mathscr{F}^+$-module is isomorphic to one of $\mathscr{F}^\pm,\,\mathscr{F}^\pm_t$.
\end{enumerate}
\end{theorem}

\section{Generators of bimodules for $A(\mathscr{F}^+)$ for $d=1$ case}\label{sect-4}
In this section we will find generators of bimodules for 
$A(\mathscr{F}^+)$ for $d=1$.

\subsection{Zhu's algebra $A(\mathscr{F}^+)$}\label{subsect-4-1}
Let us recall from \cite{Abe07} the generators of the Zhu's algebra of $\mathscr{F}^+$.

By \cite[Corollary 3.8]{Abe07} it follows that the vertex operator algebra
$\mathscr{F}^+$ is strongly generated by the Virasoro element 
$\w=\sum_{i=1}^de_{(-1)}^if^i$ and the elements
\begin{align*}
&e^{ij}=e^{i}_{(-1)}e^{j},\quad h^{ij}=e^{i}_{(-1)}f^j,\quad 
f^{ij}=f^{i}_{(-1)}f^j,\\
&E^{ij}=\frac{1}{2}\big(e^{i}_{(-2)}e^{j}+e^{j}_{(-2)}e^{i}\big),\\
&H^{ij}=\frac{1}{2}\big(e^{i}_{(-2)}f^j+f^{j}_{(-2)}e^i\big),\\
&F^{ij}=\frac{1}{2}\big(f^{i}_{(-2)}f^j+f^{j}_{(-2)}f^i\big) 
\end{align*}
for  $1\leq i,\,j \leq d$. Hence, Zhu's algebra $A(\mathscr{F})$ is generated by the set 
\begin{equation*}
\{e^{ij},\, h^{ij},\, f^{ij},\, E^{ij},\, H^{ij},\, F^{ij}\}_{i,j=1}^d
\end{equation*}
(see \cite[Proposition 2.5]{Abe07}).
We denote $\mathscr{F},\,\mathscr{F}_t,\,\mathscr{F}^\pm$ and
$\mathscr{F}^\pm_t$ for $d=1$ by $\mathscr{T},\,\mathscr{T}_t,\,
\mathscr{T}^\pm$ and $\mathscr{T}^\pm_t$ , respectively.
In the case that $d=1$ we see that  $e^{1,1}=f^{1,1}=0,\, \w=h^{11}$. We set $E=E^{11}$, $H=H^{11}$ and $F=F^{11}$.
Therefore Zhu's algebra $A(\mathscr{T}^+)$ is generated by $\w$, $E$, $H$ and $F$. 

\subsection{Generators of the $A(\mathscr{T}^+)$-bimodule $A(\mathscr{T}^-)$}\label{subsect-4-2}
Let use set
\begin{equation*}
U=\langle a^1\ast a^2\ast\dotsb\ast a^r \ast h\ast b^1\ast b^2\dotsb \ast b^s\,|\,a^i, b^j\in\mathscr{T}^+,
h\in\frakh\rangle_\C.
\end{equation*}

Let $(\mathscr{T}^-)^{(r)}$ be the vector subspace of $\mathscr{T}^-$ linearly spanned by vectors of the form 
\begin{equation*}
h^1_{(-n_1)}h^2_{(-n_2)}\dotsb h^{r}_{(-n_s)}\vac,\,h^i\in\frakh,\,n_i>0,\,s\le r.
\end{equation*}
Notice that $(\mathscr{T}^-)^{(2m+1)}=(\mathscr{T}^-)^{(2m+2)}$ for nonnegative integers $m$.
Then we see that
\begin{align}
&(h^1_{(-n_1)}h^2_{(-n_2)}\vac)_{(-1)}u\equiv h^1_{(-n_1)}h^2_{(-n_2)}u \mod (\mathscr{T}^-)^{(r)},\label{eq-0708-1}\\
&(h^1_{(-n_1)}h^2_{(-n_2)}\vac)_{(j)}u\equiv0 \mod (\mathscr{T}^-)^{(r)}
\end{align}
for all $u\in(\mathscr{T}^-)^{(r)}$, $\phi, \psi\in\frakh$, $n_1, n_2>0$
and nonnegative integers $j$.
Hence we have
\begin{equation}\label{eq-0708-3}
\begin{split}
(h^1_{(-n_1)}h^2_{(-n_2)}\vac)\ast u&=\sum_{i=0}^\infty\binom{n_1+n_2}{i}(h^1_{(-n_1)}h^2_{(-n_2)})_{(i-1)}u\\
&\equiv \phi_{(-n_1)}\psi_{(-n_2)}u \mod (\mathscr{T}^-)^{(r)}
\end{split}
\end{equation}
for all $u\in(\mathscr{T}^-)^{(r)}$, $h^1, h^2\in\frakh$ and $n_1, n_2>0$.

\begin{proposition}
The bimodule $A(\mathscr{T}^-)$ is generated by $[e]$ and $[f]$.
\end{proposition}
\begin{proof}
We show that $U+O(\mathscr{T}^-)=\mathscr{T}^-$ by using induction on $r=2m+1$ of $(\mathscr{T}^-)^{(r)}$.

For any $h_{(-n)}\vac\in(\mathscr{T}^-)^{(1)}$, we have
\begin{equation*}
h_{(-n)}\vac=\frac{1}{(n-1)!}L_{-1}^{n-1}h.
\end{equation*}
Then, by Lemma \ref{lem-1-10}, we obtain $h_{(-n)}\vac\in U+O(\mathscr{T}^-)$.

Let $m>1$ and suppose that $(\mathscr{T}^-)^{(2m+1)}\subset U+O(\mathscr{T}^-)$.
Then it follows from \eqref{eq-0708-1}--\eqref{eq-0708-3} that
\begin{equation*}
h^1_{(-n_1)}h^2_{(-n_2)}u\equiv(h^1_{(-n_1)}h^2_{(-n_2)}\vac)\ast u\mod(\mathscr{T}^-)^{(2m+1)}
\end{equation*}
for $u\in (\mathscr{T}^-)^{(2m+1)}$.
Note that $(\mathscr{T}^-)^{(2m+3)}$ is linearly spanned by vectors of the form $\phi_{(-n_1)}\psi_{(-n_2)}u$
with $u\in(\mathscr{T}^-)^{(2m+1)}$.
Therefore, by the induction hypothesis 
and the fact that $O(\mathscr{T}^-)$ is invariant under the left and right actions of $V$ with respect to
$\ast$, we have
$(\mathscr{T}^-)^{(2m+1)}\subset U+O(\mathscr{T}^-)$.
\end{proof}

\subsection{Generators of $A(\mathscr{T}^+)$-bimodules $A(\mathscr{T}_t^\pm)$}\label{subsect-4-3} 
Let 
\begin{align*}
&U_{0}^{+}=\langle\, a^{1}*\dotsb*a^{r}*\vact*b^{1}*\dotsb*b^{s}
\,|\,a^i,\,b^j\in\mathscr{T}^+\,\rangle_{\C},\\
&U_0^-
=\langle\, a^{1}*\dotsb*a^{r}*h_{(-\frac{1}{2})}\vact*b^{1}*\dotsb*b^{s}\,|
\,a^i,\,b^j\in\mathscr{T}^+, h\in\h\,\rangle_{\C}.
\end{align*}
The main result in this section is:

\begin{proposition}\label{proposiion333}  
The $A(\mathscr{T}^+)$-bimodule $A(\mathscr{T}_t^+)$ is generated by $[\vact]$ and 
the $A(\mathscr{T}^+)$-bimodule $A(\mathscr{T}_t^-)$ is generated 
by $[e_{(-\frac{1}{2})}\vact]$ and $[f_{(-\frac{1}{2})}\vact]$. 
\end{proposition}

Since $\mathscr{T}_t=\mathscr{T}_t^+\oplus\mathscr{T}_t^-$, 
it follows that $A(\mathscr{T}_{t})=A(\mathscr{T}_t^+)\oplus A(\mathscr{T}_t^-)$ as $A(\mathscr{T}^+)$-bimodules. 
Hence it is enough to find generators of $A(\mathscr{T}_t)$ as 
an $A(\mathscr{T}^+)$-bimodule. Before we give a proof 
of Proposition \ref{proposiion333} we will prepare several lemmas.

For any non-negative integers $r$ and $d$, we set 
\begin{align*}
&\mathscr{T}_{t}^{(r)}
=\langle\, h^1_{(-m_1)}\cdots h^s_{(-m_s)}\vact
\,|\,h^i\in\h,\,m_i\in \frac{1}{2}+\Zpos,\,0\leq s\leq r\,\rangle_{\C},\\
&\mathscr{T}_{t}^{+}(d)=\mathscr{T}_{t,_d-\frac{1}{8}},
\quad \mathscr{T}_{t}^{-}(d)=\mathscr{T}_{t,d+\frac{3}{8}},
\quad \mathscr{T}_{t}(d)=\mathscr{T}_{t}^+(d)\oplus\mathscr{T}_{t}^-(d)\\
&\mathscr{T}_{t}^{(r,d)}
=\bigoplus_{k=0}^d\mathscr{T}_{t}^{(r)}\cap\mathscr{T}_{t}(k)
\end{align*}
Remark that 
\begin{equation*}
\mathscr{T}_{t}^{(r)}=\sum_{d=0}^\infty \mathscr{T}_{t}^{(r,d)}
\end{equation*}
since $\mathscr{T}_{t}^{(r)}$ is $L_0$-invariant.

\begin{lemma}
{\rm (1)} For $0\leq p\leq k-1\,(k\ge1)$
and any $u\in\mathscr{T}_{t}^{(r,d)}$, we have 
\begin{align}
&(h^1_{(-p-1)}h^2_{(-k+p)}\vac)_{(-1)}u\equiv (h^1_{(-p-1)}h^2_{(-k+p)}\vac)*u
\quad \mod\mathscr{T}_{t}^{(r+2,d+k)},\label{sssd1}\\
&(h^1_{(-k-1)}h^2)_{(0)}u\equiv (h^1_{(-k-1)}h^2)*u-u*(h^1_{(-k-1)}h^2)
\quad\mod\mathscr{T}_{t}^{(r+2,d+k)}.\label{sssd2}
\end{align}
{\rm (2)} We have
\begin{align}\label{uuehje1}
&(h^1_{(-m)}h^2_{(-n)}\vac)_{(-1)}u=\sum_{\stackrel{i+j=-m-n}{i,j\in\frac{1}{2}+\Z}}
\binom{-i-1}{m-1}\binom{-j-1}{n-1}\NO h^1_{(i)}h^2_{(j)}\NO u,\\
&\label{uuehje2}
(h^1_{(-m)}h^2_{(-n)}\vac)_{(0)}u=\sum_{\stackrel{i+j=-m-n+1}{i,j\in\frac{1}{2}+\Z}}
\binom{-i-1}{m-1}\binom{-j-1}{n-1}\NO h^1_{(i)}h^2_{(j)}\NO u.
\end{align}
\end{lemma}
\begin{proof}
(1) Note that $(h^1_{(-m)}h^2_{(-n)}\vac)_{(i)}u\in\mathscr{T}_t^{(r+2,d+m+n-i-1)}$ for all integers $i$. 
Then the fact that
\begin{align*}
&a*u=a_{(-1)}u+\sum_{i=0}^\infty\binom{|a|}{i+1}a_{(i)}u,\\
&u*a=a_{(-1)}u+\sum_{i=0}^\infty\binom{|a|-1}{i+1}a_{(i)}u
\end{align*}
for any homogeneous $a\in\mathscr{T}^+$ and $u\in\mathscr{T}^+_t$ yields 
\eqref{sssd1} and \eqref{sssd2}

(2)  For any positive integers $m$ and $n$, we see that
\begin{equation*}
\begin{split}
&(h^1_{(-m)}h^2_{(-n)}\vac)_{(k)}\\
&=\Res_{z=0}z^kY^\theta(h^1_{(-m)}h^2_{(-n)}\vac,z)\\
&=\Res_{z=0}z^k\left(W(h^1_{(-m)}h^2_{(-n)}\vac,z)
+W(\Delta(z)h^1_{(-m)}h^2_{(-n)}\vac,z)\right)\\
&=\Res_{z=0}z^k\left(W(\NO\partial^{(m-1)}h^1(z)\partial^{(n-1)}h^2(z)\vac,z)
+\gamma\id\,z^{-m-n}\right) 
\end{split}
\end{equation*}
for some complex number $\gamma$.
Then the both identities immediately follow.
\end{proof}

By the lemma,  for any $0\leq p\leq k-1$ 
and $u\in \mathscr{T}_{t}^{(r,d)}$, 
one has
\begin{equation*}
\begin{split}
&(h^1_{(-p-1)}h^2_{(-k+p)}\vac)_{(-1)}u\\
&=\sum_{\stackrel{i+j=-k-1}{i,j\in\frac{1}{2}+\Z}}\binom{-i-1}{p}\binom{-j-1}{k-1-p}\NO h^1_{(i)}h^2_{(j)}\NO u\\
&\equiv \sum_{q=0}^{k}\binom{q-\frac{1}{2}}{p}\binom{k-1-q+\frac{1}{2}}{k-1-p} 
h^1_{(-q-\frac{1}{2})}h^2_{(-k+q-\frac{1}{2})} u\mod\mathscr{T}_{t}^{(r,d+k+1)}\\
&=\sum_{q=0}^{k}A^{k-1}_{p,q}\bigl(-\frac{1}{2},\frac{1}{2},-\frac{1}{2}\bigr)h^1_{(-q-\frac{1}{2})}
h^2_{(-k+q-\frac{1}{2})} u
\end{split}
\end{equation*}
and
\begin{equation*}
\begin{split}
&(h^1_{(-k-1)}h^2)_{(0)}u\\
&= \sum_{\stackrel{i+j=-k-1}{i,j\in\frac{1}{2}+\Z}}
\binom{-i-1}{k}\NO h^1_{(i)}h^2_{(j)}\NO u\\
&\equiv \sum_{q=0}^{k}\binom{q-\frac{1}{2}}{k}h^1_{(-q-\frac{1}{2})}h^2_{(-k+q-\frac{1}{2})}u
\mod\mathscr{T}_{t}^{(r,d+k+1)}\\
&=\sum_{q=0}^{k}A^{k-1}_{k,q}\bigl(-\frac{1}{2},\frac{1}{2},-\frac{1}{2}\bigr)h^1_{(-q-\frac{1}{2})}
h^2_{(-k+q-\frac{1}{2})}u
\end{split}
\end{equation*}
where we set the matrix 
$A^k(a,b,c)=(A^{k}_{p,q}(a,b,c))_{0\le p,q\le k+1}$
with
\begin{equation*}
A^{k}_{p,q}(a,b,c)=
\begin{cases}
\displaystyle{\binom{q+a}{p}\binom{k-q+b}{k-p}}&\quad
(0\leq p\leq k,\, 0\leq q\leq k+1),\\
&\\
\displaystyle{\binom{q+c}{k+1}}
&\quad (p=k+1,\,0\leq q\leq k+1)
\end{cases}
\end{equation*}

Since the matrix $A^{k}\bigl(-\frac{1}{2},\frac{1}{2},-\frac{1}{2}\bigr)$ is invertible 
by Lemma \ref{theorem01}, \eqref{sssd1}--\eqref{sssd2} imply:
 
\begin{lemma}\label{aaiw}
For any $k\in\Zplus$, $0\leq q\leq k$, $r\in\Zpos$, $h^1, h^2\in\h$ and 
$u\in\mathscr{T}_t^{(r,d)}$, the element 
$h^1_{(-q-\frac{1}{2})}h^2_{(-k+q-\frac{1}{2})}u$ is in the subspace
$\mathscr{T}^+*u*\mathscr{T}^++\mathscr{T}_{t}^{(r,d+k+1)}+\mathscr{T}_{t}^{(r+2,d+k)}$.
\end{lemma} 

\underline{{\it Proof of Proposition \ref{proposiion333} 
for $\mathscr{T}_t^+$}.} 
We shall show that $\mathscr{T}_t^+(d)\subset U_0^+$ for any 
non-negative integer $d\geq 0$ by using induction on $d$.

The case that $d=0$ is obvious because $\mathscr{T}_t^+(0)=\C\vact$. 
For  $d=1$, we have $\mathscr{T}_t^+(1)=\C L_{-1}\vact\subset U_0$ 
since $L_{-1}\vact=\w*\vact-\vact*\w+\frac{1}{8}\vact \in U_0^+$

Let $d>1$ and suppose that $\mathscr{T}_t^+(d_1)\subset U_0^+$ for any $d_1<d$.  Then we see that any element 
in $\mathscr{T}_t^+(d)$ is a linear combination of elements of the form 
\begin{equation*}
v=h^1_{(-p-\frac{1}{2})}h^2_{(-k+p-\frac{1}{2})}u
\end{equation*}
for $h^1,\, h^2\in\h$, $1\leq k\leq d-1$, $0\leq p \leq k-1$ and 
$u\in\mathscr{T}_{t}^{(r-2)}\cap \mathscr{T}_{t}^+(d-k-1)$ with any
$r\in 2\Zplus$.  Then by Lemma \ref{aaiw} and induction hypothesis, 
the element $v$ is in $\mathscr{T}_{t}^{(r-2)}\cap \mathscr{T}_{t}^+ +U_0^+$. 
Now induction on $m\geq 1$ for $r=2m$
and above discussion imply that $v\in U_0^+$. 
This proves $\mathscr{T}_t^+(d)\subset U_0^+$.\hfill$\Box$\vskip5mm
 
\underline{{\it Proof of Proposition \ref{proposiion333} for $\mathscr{T}_t^-$}.} 
We also use induction on $d$. In the case $d=0$,  we obviously have
$\mathscr{T}_t^-(0)=\C e_{(-\frac{1}{2})}\vact+\C f_{(-\frac{1}{2})}\vact\subset U_0^-$.
For $d=1$, we see that 
$\mathscr{T}_t^-(1)=\C e_{(-\frac{3}{2})}\vact+\C f_{(-\frac{3}{2})}\vact$.
Since 
\begin{equation*}
\frac{1}{2}h_{(-\frac{3}{2})}\vact=
L_{-1}h_{(-\frac{1}{2})}\vact=\w*h_{(-\frac{1}{2})}\vact
-h_{(-\frac{1}{2})}\vact*\w-\frac{3}{8}h_{(-\frac{1}{2})}\vact\in U_0^-
\end{equation*}
for $h=e$ and $f$ by Lemma \ref{prop-1-2}, we have $\mathscr{T}_t^-(1)\subset U_0^{-}$. 

Let $d>1$ and suppose that $\mathscr{T}_t^-(d_1)\subset U_0^-$ for $d_1<d$. We see that any element of $\mathscr{T}_t^-(d)$ 
with  length $r\geq 3$ is a linear combination of elements of the form 
\begin{equation*}
v=h^1_{(-p-\frac{1}{2})}h^2_{(-k+p-\frac{1}{2})}u
\end{equation*}
for $h^1, h^2\in \h$, $1\leq k\leq d-1$, $0\leq p \leq k-1$ and 
$u\in\mathscr{T}_{t}^{(r-2)}\cap \mathscr{T}_{t}^-(d-k-1)$ with $r\in1+2\Zplus$. 
By Lemma \ref{aaiw} and induction hypothesis, the element $v$ 
belongs to $\mathscr{T}_{t}^{(r-2,d)}\cap\mathscr{T}_{t}^-+U_0^-$. 
Now induction on $m\geq 1$ for $r=2m+1$ yields that 
$v\in \mathscr{T}_{t}^{(1)}\cap \mathscr{T}_{t}^-(d)+U_0^-$. 
This proves 
$\mathscr{T}_t^-(d)\subset\mathscr{T}_{t}^{(1)
}\cap\mathscr{T}_{t}^-(d)+U_0^-$.

Now it is enough to prove 
\begin{equation}\label{asidufh}
\mathscr{T}_{t}^{(1)}\cap \mathscr{T}_{t}^-(d)
=\C e_{(-\frac{1}{2}-d)}\vact
+\C f_{(-\frac{1}{2}-d)}\vact\subset U_0^-.
\end{equation}
It follows from \eqref{sssd1}--\eqref{sssd2} together with induction hypothesis on $d$ that  
\begin{align*}
(e_{(-1-p)}f_{(-d+p+1)}\vac)_{(-1)}e_{(-\frac{1}{2})}\vact,\quad 
(e_{(-d)}f)_{(0)}e_{(-\frac{1}{2})}\vact\in U_0^-
\end{align*}
for any $0\leq p\leq d-2$. 

On the other hand, by \eqref{uuehje1}--\eqref{uuehje2}, we have 
\begin{align*}
&(e_{(-p-1)}f_{(-d+p+1)}\vac)_{(-1)}e_{(-\frac{1}{2})}\vact\\
&=\sum_{q=1}^{d-1}A^{d-2}_{p,q}\bigl(-\frac{1}{2},\frac{1}{2},-\frac{1}{2}\bigr)
e_{(-q-\frac{1}{2})}f_{(-d+q+\frac{1}{2})}e_{(-\frac{1}{2})}\vact\\
&\qquad\qquad+\frac{1}{2}A^{d-2}_{d-1,d}\bigl(-\frac{1}{2},\frac{1}{2},-\frac{1}{2}\bigr)
e_{(-\frac{1}{2}-d)}\vact,\\
&=\sum_{q=0}^{d-2}A^{d-2}_{p,q+1}\bigl(-\frac{1}{2},\frac{1}{2},-\frac{1}{2}\bigr)
e_{(-q-\frac{3}{2})}f_{(-d+q-\frac{1}{2})}e_{(-\frac{1}{2})}\vact\\
&\qquad\qquad+\frac{1}{2}A^{d-2}_{d-1,d}\bigl(-\frac{1}{2},\frac{1}{2},-\frac{1}{2}\bigr)e_{(-\frac{1}{2}-d)}\vact,\\
\intertext{and}
&(e_{(-d)}f)_{(0)}e_{(-\frac{1}{2})}\vact\\
&=\sum_{q=1}^{d-1}A^{d-2}_{d-1,q}\bigl(-\frac{1}{2},\frac{1}{2},-\frac{1}{2}\bigr)e_{(-q-\frac{1}{2})}f_{(-d+q+\frac{1}{2})}e_{(-\frac{1}{2})}\vact\\
&\qquad\qquad+A^{d-2}_{d-1,d}\bigl(-\frac{1}{2},\frac{1}{2},-\frac{1}{2}\bigr)\frac{1}{2}e_{(-\frac{1}{2}-d)}\vact\\
&=\sum_{q=0}^{d-2}A^{d-2}_{d-1,q+1}\bigl(-\frac{1}{2},\frac{1}{2},-\frac{1}{2}\bigr)e_{(-q-\frac{3}{2})}
f_{(-d+q-\frac{1}{2})}e_{(-\frac{1}{2})}\vact\\
&\qquad\qquad+A^{d-2}_{d-1,d}\bigl(-\frac{1}{2},\frac{1}{2},-\frac{1}{2}\bigr)\frac{1}{2}e_{(-\frac{1}{2}-d)}\vact.
\end{align*}
By the definition of matrices $A$ we see that for $0\leq p\leq d-2$ 
and $0\leq q\leq d-2$,
\begin{equation*}
\begin{split}
A^{d-2}_{p,q+1}\bigl(-\frac{1}{2},\frac{1}{2},-\frac{1}{2}\bigr)&=\binom{q+\frac{1}{2}}{p}\binom{d-2-q-\frac{1}{2}}{d-2-p}
\\
&=A_{p,q}^{d-2}\bigl(\frac{1}{2},-\frac{1}{2},\frac{1}{2}\bigr)
\end{split}
\end{equation*}
and 
\begin{equation*}
\begin{split}
A^{d-2}_{d-1,q+1}(-\frac{1}{2},\frac{1}{2},-\frac{1}{2})&=\binom{q+\frac{1}{2}}{d-1}\\
&=A^{d-2}_{d-1,q}\bigl(\frac{1}{2},-\frac{1}{2},\frac{1}{2}\bigr).
\end{split}
\end{equation*}
Since  $A^{d-2}(\frac{1}{2},-\frac{1}{2},\frac{1}{2})$ is invertible 
because $\det A^{d-2}(\frac{1}{2},-\frac{1}{2},\frac{1}{2})=1$ by Theorem \ref{theorem01}, 
it follows that $e_{(-q-\frac{3}{2})}f_{(-d+q-\frac{1}{2})}e_{(-\frac{1}{2})}\vact$, 
$e_{(-\frac{1}{2}-d)}\vact\in U_0^-$ for any $0\leq q\leq d-2$.  

In the same way we can show that 
$f_{(-q-\frac{3}{2})}e_{(-d+q-\frac{1}{2})}f_{(-\frac{1}{2})}\vact$, 
$f_{(-\frac{1}{2}-d)}\vact\in U_0^-$ for any $0\leq q\leq d-2$.
Therefore we have \eqref{asidufh}. \hfill$\Box$\vskip5mm

\section{Fusion rules for simple $\mathscr{F}^+$-modules}\label{sect-5}

In this section we will first determine fusion rules for simple $\mathscr{T}^+$-modules and then simple $\mathscr{F}^+$-modules.

\subsection{The case $d=1$}\label{subsect-5-1}
In this section we shall determine fusion rules for $\mathscr{F}^+$ 
with $d=1$, i.e., for $\mathscr{T}^+$.
In order to determine fusion rules for $\mathscr{F}^+$ with general $d$
it is enough to consider $d=1$ case as we will explain in the next subsection.

By Proposition \ref{prop-1-5} we have:
\begin{proposition}\label{prop-main-1}
The fusion rule of type $\fusion{\mathscr{T}^+}{M}{N}$ is given by
\begin{equation*}
\dim_\C I_{\mathscr{T}^+}\fusion{\mathscr{T}^+}{M}{N}=
\begin{cases}
1&\quad(M,N)=(\mathscr{T}^\pm,\mathscr{T}^\pm), (\mathscr{T}^\pm_t,\mathscr{T}^\pm_t),\\
0&\quad\text{otherwise}.
\end{cases}
\end{equation*}
\end{proposition}

Since any simple $\mathscr{T}^+$-modules is self-dual, 
Propositions \ref{prop-1-2} and \ref{prop-main-1}
yield that it is enough to calculate fusion rules 
$\dim_\C I_{\mathscr{T}^+}\fusion{L}{M}{N}$ for triples
\begin{equation}\label{fusion-triple}
\begin{split}
(L,M,N)
=&(\mathscr{T}^-,\mathscr{T}^-,\mathscr{T}^-),\,
(\mathscr{T}^-,\mathscr{T}^-,\mathscr{T}^+_t),\,
(\mathscr{T}^-,\mathscr{T}^-,\mathscr{T}^-_t),\,
(\mathscr{T}^-,\mathscr{T}^+_t,\mathscr{T}^+_t),\\
&(\mathscr{T}^-,\mathscr{T}^+_t,\mathscr{T}^-_t),\,
(\mathscr{T}^-,\mathscr{T}^-_t,\mathscr{T}^-_t),\,
(\mathscr{T}^+_t,\mathscr{T}^+_t,\mathscr{T}^+_t),\,
(\mathscr{T}^+_t,\mathscr{T}^+_t,\mathscr{T}^-_t),\\
&(\mathscr{T}^+_t,\mathscr{T}^-_t,\mathscr{T}^-_t),\,
(\mathscr{T}^-_t,\mathscr{T}^-_t,\mathscr{T}^-_t).
\end{split}
\end{equation}

Let us first determine fusion rules of type $\fusion{\mathscr{T}^-}{M}{N}$.
For any $h\in\frakh$, we have
\begin{equation}\label{eq-arike-1}
\omega\ast h=L_{-2}h+2L_{-1}h+L_0h
=\frac{1}{2}L_{-1}^2h+2L_{-1}h+h.
\end{equation}
Applying Lemma \ref{lem-1-10} to \eqref{eq-arike-1} we obtain
\begin{equation*}
\begin{split}
\omega \ast h&\equiv
\frac{1}{2}\omega\ast  L_{-1}h-\frac{1}{2}(L_{-1}h) \ast\omega\\
&-L_{-1}h+2\omega\ast  h-2h\ast \omega-2h\mod O(\mathscr{T}^-)\\
&\equiv\frac{1}{2}\omega\ast(\omega\ast h-h\ast \omega-h)\\
&\quad-\frac{1}{2}(\omega\ast h-h\ast \omega-h)\ast 
\omega\\
&\quad\quad-(\omega \ast h-h\ast \omega-h)
+2\omega\ast h-2h\ast \omega-2h+h\mod O(\mathscr{T}^-)\\
&=
\frac{1}{2}(\omega^2\ast h-2\omega\ast h \ast\omega+h \ast\omega^2)
+\frac{1}{2}(\omega\ast  h-h \ast\omega),
\end{split}
\end{equation*}
which implies 
\begin{equation}\label{eq-equation-eigenvalue}
\frac{1}{2}(\omega^2\ast  h-2\omega\ast h\ast\omega+h
\ast \omega^2)
-\frac{1}{2}(\omega\ast  h+h\ast \omega)\equiv0 \mod O(\mathscr{T}^-).
\end{equation}

For simple $\mathscr{F}^+$-modules $M$ and $N$,
we see that
\begin{equation*}
u^\prime\otimes [\omega]^m\ast [a]\ast[\omega]^n\otimes v=x^my^n u^\prime\otimes [a]\otimes v,\quad
(a\in\mathscr{T}^-,\,v\in\Omega(M),\,u^\prime\in\Omega(N)^\ast)
\end{equation*}
for nonnegative integers $m$ and $n$
where $x$ and $y$ are conformal weights of $N$ and $M$, respectively.
Then, by \eqref{eq-equation-eigenvalue}, 
setting 
$\alpha(x,y)=\frac{(x-y)^2}{2}-\frac{(x+y)}{2}$, we have
\begin{equation*}
\alpha(x,y)u^\prime\otimes [h]\otimes v=0
\end{equation*}
for any $h\in\frakh$, $u^\prime\in\Omega(N)^\ast$, and $v\in \Omega(M)$.
Recall that  $A(\mathscr{T}^-)$ is generated by 
$e$ and $f$ as an $A(\mathscr{T}^+)$-module. Therefore we have 
$\Omega(N)^\ast\cdot A(\mathscr{T}^-)\cdot \Omega(M)=0$ 
if $\alpha(x,y)\neq0$.
By direct calculations we get 
\begin{equation*}
\alpha(x,y)=
\begin{cases}
-1,&M=N=\mathscr{T}^-,\\
\frac{25}{128},&(M,N)=(\mathscr{T}^-,\mathscr{T}_t^+),\\
-\frac{63}{128},&(M,N)=(\mathscr{T}^-,\mathscr{T}_t^-),\\
\frac{1}{8},&M=N=\mathscr{T}^+_t,\\
0,&(M,N)=(\mathscr{T}^+_t,\mathscr{T}_t^-),\\
-\frac{3}{8},&M=N=\mathscr{T}^-_t.
\end{cases}
\end{equation*}
These imply that
\begin{equation*}
\dim_\C I_{\mathscr{T}^+}\fusion{\mathscr{T}^-}{M}{N}=0
\quad\text{for}\quad (M,N)\neq(\mathscr{T}^+_t,\mathscr{T}^-_t).
\end{equation*}

We have a non-trivial intertwining operator of type 
$\fusion{\mathscr{T}^-}{\mathscr{T}^+_t}{\mathscr{T}^-_t}$ which
is obtained as the restriction of the twisted vertex operator 
$Y^\theta(-,z)$
to $\mathscr{T}^-$. 
Hence the fusion rule of type $\fusion{\mathscr{T}^-}{\mathscr{T}^+_t}{\mathscr{T}^-_t}$ is non-zero.

We will prove that this fusion rule is  one-dimensional.
We need to prove that 
$\dim_\C\Omega(\mathscr{T}_t^-)^*\cdot A(\mathscr{T}^-)
\cdot\Omega(\mathscr{T}^+_t)=1$ 
by using the explicit description of right and left actions of $E=e_{(-2)}e$ and $F=f_{(-2)}f$
on $\Omega(\mathscr{T}_t^-)^*$ and $A(\mathscr{T}^-)$.

Let us consider the exact forms of the actions of $E$, $F$ on 
$\Omega(\mathscr{T}^+_t)$ and $\Omega(\mathscr{T}^-_t)^\ast$,
respectively.
Let $\{\phi_t,\psi_t\}\subset\Omega(\mathscr{T}^-_t)^*$ 
be the dual basis
 with respect to the basis
$\bigl\{e_{(-\frac{1}{2})}\vact,f_{(-\frac{1}{2})}\vact\bigr\}$
of $\Omega(\mathscr{T}^-_t)$.
Note that
\begin{equation}\label{eq-1-11}
\begin{split}
(h^1_{(-2)}h^1)_{(2)}h_{(-\frac{1}{2})}^2\vact&=\sum_{i+j=0}(-i-1)\NO h^1_{(i)}h^1_{(j)}\NO h^2_{(-\frac{1}{2})}\vact\\
&=-\langle h^1,h^2\rangle h^2_{(-\frac{1}{2})}\vact
\end{split}
\end{equation}
for any $h^1,\,h^2\in\frakh$. Thus we have
\begin{align*}
&o(E)e_{(-\frac{1}{2})}\vact=0,
\quad
o(E)f_{(-\frac{1}{2})}\vact=e_{(-\frac{1}{2})}\vact,
\\
&o(F)e_{(-\frac{1}{2})}\vact=-f_{(-\frac{1}{2})}\vact,
\quad o(F)f_{(-\frac{1}{2})}\vact=0,
\end{align*}
and hence we obtain
\begin{align}
&\phi_t\cdot o(E)=\psi_t,\quad \phi_t\cdot o(F)=0,\label{eq-1-18}\\
&\psi_t\cdot o(E)=0,\quad\;\; \psi_t\cdot o(F)=-\phi_t.\label{eq-1-19}
\end{align}
We further have $o(h^1_{(-2)}h^1)\vact=0$ by the direct calculation 
(cf. \eqref{eq-1-11}).
Since we have $E\ast e=F\ast f=0$,
\eqref{eq-1-18} and \eqref{eq-1-19} show
\begin{align*}
&\psi_t\otimes [e]\otimes\vact=\phi_t\otimes [E\ast e]\otimes\vact
=0,
\\
&\phi_t\otimes [f]\otimes\vact=-\psi_t\otimes [F\ast f]\otimes\vact
=0,
\end{align*}

Let us find several relations in $A(\mathscr{T}^-)$.
Direct calculations show
\begin{equation}\label{eq-cont-mi-1}
\begin{split}
&(h^1_{(-2)}h^1)*h^2-h^2*(h^1_{(-2)}h^1)\\
&=
4\langle h^1,h^2\rangle h^1_{(-3)}\vac
+8\langle h^1,h^2\rangle h^1_{(-2)}\vac
+2\langle h^1,h^2\rangle h^1\\
&=2\langle h^1,h^2\rangle L_{-1}^2h^1+8\langle h^1,h^2\rangle L_{-1}h^1+2\langle h^1,h^2\rangle h^1\\
\end{split}
\end{equation}
Since $\mathscr{T}_t^+$ and $\mathscr{T}^-_t$ have conformal weights
$-\frac{1}{8}$ and $\frac{3}{8}$, respectively,
Lemma \ref{lem-1-10} and \eqref{eq-cont-mi-1} imply
\begin{equation}\label{eq-cont-mi-2}
u\otimes\bigl[(h^1_{(-2)}h^1)\ast h^2-h^2\ast(h^1_{(-2)}h^1)
+\frac{1}{2}\langle h^1,h^2\rangle h^1\bigr]
\otimes v=0
\end{equation}
for any $u^\prime\in\Omega(\mathscr{T}^-_t)^*$ and 
$v\in\Omega(\mathscr{T}^+_t)$.
By \eqref{eq-cont-mi-2} and \eqref{eq-1-18} we obtain
\begin{equation*}
\begin{split}
\psi_t\otimes [f]\otimes \vact&=\phi_t\otimes [E\ast f]\otimes\vact\\
&=\phi_t\otimes[E\ast f-f\ast E]\otimes\vact\\
&=\frac{1}{2}\phi_t\otimes [e]\otimes\vact.
\end{split}
\end{equation*}
Therefore the vector space 
$\Omega(\mathscr{T}^-_t)^*\cdot 
A(\mathscr{T}^-)\cdot\Omega(\mathscr{T}^+_t)$
is spanned by the element $\phi_t\otimes [e]\otimes\vact$, 
which yields that the fusion rule of type 
$\fusion{\mathscr{T}^-}{\mathscr{T}_t^+}{\mathscr{T}_t^-}$ 
is equal to one.

Summarizing we have:
\begin{proposition}\label{prop-2-2}
For any triple $(\mathscr{T}^-,M,N)$ in \eqref{fusion-triple}
with $L=\mathscr{T}^-$,
we have
\begin{equation*}
\dim_\C I_{\mathscr{T}^+}\fusion{\mathscr{T}^-}{M}{N}=
\begin{cases}
1&\quad(M,N)=(\mathscr{T}^+_t,\mathscr{T}^-_t),\\
0&\quad\text{otherwise}.
\end{cases}
\end{equation*}
\end{proposition}

To complete the list \eqref{fusion-triple} we will prove the following proposition.

\begin{proposition}\label{prop-2-3}
\noindent
{\rm (1)}
For  $(M,N)=(\mathscr{T}^+_t,\mathscr{T}^+_t),(\mathscr{T}^+_t,\mathscr{T}^-_t),
(\mathscr{T}^-_t,\mathscr{T}^-_t)$,
we have
\begin{equation*}
\dim_\C I_{\mathscr{T}^+}\fusion{\mathscr{T}^+_t}{M}{N}=
0
\end{equation*}
\noindent
{\rm (2)}
We have
\begin{equation*}
\dim_\C I_\mathscr{T}^+\fusion{\mathscr{T}^-_t}{\mathscr{T}^-_t}{\mathscr{T}^-_t}=0.
\end{equation*}
\end{proposition}
\begin{proof}
\noindent
(1)
Recall from Theorem \ref{proposiion333} that $A(\mathscr{T}^+_t)$ is generated by $[\vact]$ as an $A(\mathscr{t}^+)$-module.
Therefore it is enough to show
$\varphi\otimes[\vact]\otimes u=0$ for all $\varphi\in \Omega(N)^\ast$ and $u\in \Omega(M)$.
Since
\begin{align*}
&L_{-1}\vact=e_{(-\frac{1}{2})}f_{(-\frac{1}{2})}\vact,
\\
&L_{-2}\vact=e_{(-\frac{1}{2})}f_{(-\frac{3}{2})}\vact+e_{(-\frac{3}{2})}f_{(-\frac{1}{2})}\vact,
\end{align*}
we obtain
\begin{equation*}
L_{-2}\vact=2L_{-1}^2\vact.
\end{equation*}
Note that
\begin{equation*}
\vact\ast\omega=L_{-2}\vact+L_{-1}\vact=2L^2_{-1}\vact+L_{-1}\vact.
\end{equation*}
Let us set 
\begin{equation*}
\beta(x,y)=y-2\Bigl(x-y-\frac{7}{8}\Bigr)\Bigl(x-y+\frac{1}{8}\Bigr)-\Bigl(x-y+\frac{1}{8}\Bigr).
\end{equation*}
Then by Lemma \ref{lem-1-10}, for simple $\mathscr{T}^+$-modules $M$ and $N$,
we have
\begin{equation}\label{eq-cont-4}
\beta(x,y)u^\prime\otimes [\vact]\otimes v=0 \quad (u^\prime\in\Omega(N)^\ast,\,v\in\Omega(M))
\end{equation}
where
$x$ and $y$ are conformal weights of $N$ and $M$, respectively.
It is not difficult to see that 
\begin{equation}\label{eq-cont-5}
\beta(x,y)=
\begin{cases}
-\frac{9}{32},&\quad (M,N)=(\mathscr{T}^+_t,\mathscr{T}^-_t)\quad(x,y)=(\frac{3}{8},-\frac{1}{8}),\\
-\frac{1}{32}&\quad M=N=\mathscr{T}_t^+\quad(x,y)=(-\frac{1}{8},-\frac{1}{8}),\\
\frac{15}{32}&\quad M=N=\mathscr{T}_t^-\quad(x,y)=(\frac{3}{8},\frac{3}{8}).
\end{cases}
\end{equation}
Hence it follows from \eqref{eq-cont-4} and \eqref{eq-cont-5} that
$\varphi\otimes[\vact]\otimes u=0$ 
for all $\varphi\in\Omega(N)^*$ and $u\in\Omega(M)$.
Therefore we have shown (1).

(2)
By using the fact that $A(\mathscr{T}_t^-)$ is generated by $[e_{(-\frac{1}{2})}\vact]$ and $[f_{(-\frac{1}{2})}\vact]$ 
as an $A(\mathscr{T}_t^+)$-module
by Theorem \ref{proposiion333}, we  prove
$\varphi\otimes [h_{(-\frac{1}{2})}\vact]\otimes u=0\,(h=e,f)$ for all
$\varphi\in\Omega(\mathscr{T}_t^-)^*$ and $u\in\Omega(\mathscr{T}_t^-)$.
To prove this we need  identities which can be verified by direct 
calculations.
Recall that $H=\frac{1}{2}(e_{(-2)}f-f_{(-2)}e)$. Then we have:
\begin{align}
&
H\ast h_{(-\frac{1}{2})}\vact-(h_{(-\frac{1}{2})}\vact)\ast H
=H_{(0)}h_{(-\frac{1}{2})}\vact+2H_{(1)}h_{(-\frac{1}{2})}\vact\notag\\
&\qquad\qquad\qquad\qquad\qquad\qquad\qquad\qquad+H_{(2)}h_{(-\frac{1}{2})}\vact,\label{eq-cont-t-0}
\\
&H_{(0)}h_{(-\frac{1}{2})}\vact=\frac{1}{2}f_{(-\frac{3}{2})}e_{(-\frac{1}{2})}h_{(-\frac{1}{2})}\vact
-\frac{1}{2}e_{(-\frac{3}{2})}h_{(-\frac{1}{2})}f_{(-\frac{1}{2})}\vact\notag\\
&\qquad\qquad\qquad\qquad+\frac{3}{4}\langle f,h\rangle
e_{(-\frac{5}{2})}\vact+\frac{3}{4}\langle e,h\rangle f_{(-\frac{5}{2})}\vact,
\label{eq-cont-t-1}\\
&H_{(1)}h_{(-\frac{1}{2})}\vact=\frac{1}{2}\langle f,h\rangle e_{(-\frac{3}{2})}\vact+\frac{1}{2}\langle e,h\rangle
f_{(-\frac{3}{2})}\vact,\label{eq-cont-t-2}\\
&H_{(2)}h_{(-\frac{1}{2})}\vact=\frac{1}{4}\langle f,h\rangle e_{(-\frac{1}{2})}\vact+\frac{1}{4}\langle e,h\rangle
f_{(-\frac{1}{2})}\vact,\label{eq-cont-t-3}\\
&(h_{(-\frac{1}{2})}\vact)\ast\omega=L_{-2}h_{(-\frac{1}{2})}\vact+L_{-1}h_{(-\frac{1}{2})}\vact.
\label{eq-cont-t-4}
\end{align}
Since  
\begin{align}
&L_{-1}h_{(-\frac{1}{2})}\vact=\frac{1}{2}h_{(-\frac{3}{2})}\vact,\label{eq-cont-t-5}\\
&L_{-2}h_{(-\frac{1}{2})}\vact =\frac{1}{2}h_{(-\frac{5}{2})}\vact
-f_{(-\frac{3}{2})}e_{(-\frac{1}{2})}h_{(-\frac{1}{2})}\vact
-e_{(-\frac{3}{2})}h_{(-\frac{1}{2})}f_{(-\frac{1}{2})}\vact\label{eq-cont-t-6}
\end{align}
for $h=e, f$,
we conclude that 
\begin{equation}\label{eq-cont-t-6-1}
L_{-1}^2h_{(-\frac{1}{2})}\vact=\frac{3}{4}h_{(-\frac{5}{2})}\vact
+\frac{1}{2}h_{(-\frac{3}{2})}e_{(-\frac{1}{2})}f_{(-\frac{1}{2})}\vact.
\end{equation}

On the one hand,
we have by \eqref{eq-cont-t-0}--\eqref{eq-cont-t-4} that
\begin{equation}\label{eq-cont-t-7}
\begin{split}
&\frac{4}{3}(H\ast e_{(-\frac{1}{2})}\vact-(e_{(-\frac{1}{2})}\vact)\ast H)
-(e_{(-\frac{1}{2})}\vact)\ast\omega\\
&=-\frac{2}{3}e_{(-\frac{3}{2})}e_{(-\frac{1}{2})}f_{(-\frac{1}{2})}\vact
+e_{(-\frac{5}{2})}\vact+\frac{4}{3}e_{(-\frac{3}{2})}\vact+\frac{1}{3}e_{(-\frac{1}{2})}\vact\\
&\qquad-L_{-2}e_{(-\frac{1}{2})}\vact-L_{-1}e_{(-\frac{1}{2})}\vact.
\end{split}
\end{equation}
On the other hand, from \eqref{eq-cont-t-5}--\eqref{eq-cont-t-6-1}  the right-hand side of \eqref{eq-cont-t-7}
becomes
\begin{equation*}
\begin{split}
&-\frac{2}{3}e_{(-\frac{3}{2})}e_{(-\frac{1}{2})}f_{(-\frac{1}{2})}\vact
+e_{(-\frac{5}{2})}\vact+\frac{4}{3}e_{(-\frac{3}{2})}\vact+\frac{1}{3}e_{(-\frac{1}{2})}\vact\\
&-\frac{1}{2}e_{(-\frac{5}{2})}\vact+e_{(-\frac{3}{2})}e_{(-\frac{1}{2})}f_{(-\frac{1}{2})}\vact
-L_{-1}e_{(-\frac{1}{2})}\vact\\
&=\frac{2}{3}L_{-1}^2e_{(-\frac{1}{2})}\vact+\frac{5}{3}L_{-1}e_{(-\frac{1}{2})}\vact
+\frac{1}{3}e_{(-\frac{1}{2})}\vact.
\end{split}
\end{equation*}
Hence we have
\begin{equation}\label{eq-cont-t-0705}
\begin{split}
&\frac{4}{3}(H\ast e_{(-\frac{1}{2})}\vact-(e_{(-\frac{1}{2})}\vact)\ast H)
-(e_{(-\frac{1}{2})}\vact)\ast\omega\\
&\qquad\qquad-\frac{2}{3}L_{-1}^2e_{(-\frac{1}{2})}\vact-\frac{5}{3}L_{-1}e_{(-\frac{1}{2})}\vact
-\frac{1}{3}e_{(-\frac{1}{2})}\vact=0.
\end{split}
\end{equation}
Changing $(e,f)$ by $(-f,e)$ in the above discussion we obtain
the identity
\begin{equation}\label{eq-cont-t-0705-2}
\begin{split}
&\frac{4}{3}(H\ast f_{(-\frac{1}{2})}\vact-(f_{(-\frac{1}{2})}\vact)\ast H)
+(f_{(-\frac{1}{2})}\vact)\ast\omega\\
&\qquad\qquad+\frac{2}{3}L_{-1}^2f_{(-\frac{1}{2})}\vact+\frac{5}{3}L_{-1}f_{(-\frac{1}{2})}\vact
+\frac{1}{3}f_{(-\frac{1}{2})}\vact=0.
\end{split}
\end{equation}

Note that
\begin{equation*}
\phi_t\cdot o(H)=\frac{1}{4}\phi_t,\quad \psi_t\cdot o(H)=-\frac{1}{4}\psi_t
\end{equation*}
by \eqref{eq-cont-t-3}.
Then, by \eqref{eq-cont-t-0705}, Lemma \ref{lem-1-10} and \eqref{eq-cont-t-3},
we see that
\begin{align}
&\left\{\gamma_1\bigl(\frac{1}{4},\frac{1}{4}\bigr)-\gamma_2\bigl(\frac{3}{8},\frac{3}{8}\bigr)\right\}\phi_t\otimes 
[e_{(-\frac{1}{2})}\vact]\otimes e_{(-\frac{1}{2})}\vact=0,\label{eq-0705-1}\\
&\left\{\gamma_1\bigl(\frac{1}{4},-\frac{1}{4}\bigr)-\gamma_2\bigl(\frac{3}{8},\frac{3}{8}\bigr)\right\}\phi_t\otimes 
[e_{(-\frac{1}{2})}\vact]\otimes f_{(-\frac{1}{2})}\vact=0,\label{eq-0705-2}\\
&\left\{\gamma_1\bigl(-\frac{1}{4},\frac{1}{4}\bigr)-\gamma_2\bigl(\frac{3}{8},\frac{3}{8}\bigr)\right\}\psi_t\otimes 
[e_{(-\frac{1}{2})}\vact]\otimes e_{(-\frac{1}{2})}\vact=0,\label{eq-0705-3}\\
&\left\{\gamma_1\bigl(-\frac{1}{4},-\frac{1}{4}\bigr)-\gamma_2\bigl(\frac{3}{8},\frac{3}{8}\bigr)\right\}\psi_t\otimes 
[e_{(-\frac{1}{2})}\vact]\otimes f_{(-\frac{1}{2})}\vact=0\label{eq-0705-4}
\end{align}
where
\begin{align*}
&\gamma_1(x,y)=x-y,\\
&\gamma_2(x,y)=y+\frac{2}{3}\Bigl(x-y-\frac{11}{8}\Bigr)
\Bigl(x-y-\frac{3}{8}\Bigr)+\frac{5}{3}\Bigl(x-y-\frac{3}{8}\Bigr)+\frac{1}{3}.
\end{align*}
It is obvious that $\gamma_2(\frac{3}{8},\frac{3}{8})=\frac{41}{96}$
and 
\begin{equation*}
\gamma_1(x,y)=
\begin{cases}
0, & x=y=\pm\frac{1}{4},\\
\frac{1}{2}, & x=\frac{1}{4}, y=-\frac{1}{4},\\
-\frac{1}{2}, & x=-\frac{1}{4}, y=\frac{1}{4}.
\end{cases}
\end{equation*}
Thus the coefficients in \eqref{eq-0705-1}--\eqref{eq-0705-4} are all non-zero,
which shows that 
$\varphi\otimes [e_{(-\frac{1}{2})}\vact]\otimes u=0$ for all $\varphi\in\Omega(\mathscr{T}^-_t)^\ast$
and $u\in\Omega(\mathscr{T}_t^-)$.

Applying the similar argument to \eqref{eq-cont-t-0705-2}, we have
$\varphi\otimes [f_{(-\frac{1}{2})}\vact]\otimes u=0$ for all $\varphi\in\Omega(\mathscr{T}^-_t)^\ast$
and $u\in\Omega(\mathscr{T}_t^-)$.
Therefore we have (2).
\end{proof}

By Propositions \ref{prop-main-1}--\ref{prop-2-3} we obtain:

\begin{theorem}\label{thm-main-1}
For any triple $(L, M, N)$ of simple $\mathscr{T}^+$-modules,
the fusion rule of type $\fusion{L}{M}{N}$ is $0$ or $1$. 
The fusion rule of type $\fusion{L}{M}{N}$ is $1$ if and only if the triple $(L,M,N)$ is any of 
\begin{equation*}
\begin{split}
&(\mathscr{T}^+,\mathscr{T}^+,\mathscr{T}^+),\quad
(\mathscr{T}^+,\mathscr{T}^-,\mathscr{T}^-),\quad
(\mathscr{T}^+,\mathscr{T}^+_t,\mathscr{T}^+_t),\\
&(\mathscr{T}^+,\mathscr{T}^-_t,\mathscr{T}^-_t),\quad
(\mathscr{T}^-,\mathscr{T}^+_t,\mathscr{T}^-_t)
\end{split}
\end{equation*}
and their permutations.
\end{theorem}

Let $\Fusion(\mathscr{T}^+)$ be the fusion algebra 
of the vertex operator algebra $\mathscr{T}^+$,
i.e., the additive group $\Fusion(\mathscr{T}^+)$ is freely generated  
by all inequivalent simple $\mathscr{T}^+$-modules
with the multiplication defined by
\begin{equation*}
M\times N=\sum_{L:\text{simple $\mathscr{T}^+$-module}}
\dim_\C I_{\mathscr{T}^+}\fusion{M}{N}{L}L
\end{equation*}
for simple $\mathscr{T}^+$-modules $M$ and $N$.
Note that the fusion product $\times$ is well-defined and is associative 
by Theorem \ref{thm-main-1}.

Let us set 
\begin{equation*}
\mathscr{T}^{(0,0)}=\mathscr{T}^+,\quad\mathscr{T}^{(1,0)}=\mathscr{T}^-,\quad\mathscr{T}^{(0,1)}=\mathscr{T}^+_t,\quad
\mathscr{T}^{(1,1)}=\mathscr{T}^-_t.
\end{equation*}
Then we have 
\begin{equation*}
\mathscr{T}^a\times \mathscr{T}^b=\mathscr{T}^{a+b}
\end{equation*}
for any $a,\,b\in \Z/2\Z\times \Z/2\Z
$ by Theorem \ref{thm-main-1}.

\begin{theorem}\label{thm-fusion-w-2}
The fusion algebra $\Fusion(\mathscr{T}^+)$ is isomorphic to the group algebra of $\Z/2\Z\times \Z/2\Z$
over $\Z$.
\end{theorem}

\subsection{Fusion rules for $d>1$}\label{subsect-5-2}
In this section we will determine fusion rules among simple 
$\mathscr{F}^+$-modules for $d>1$ by using the fusion rules for $d=1$.

Note that $(\mathscr{T}^+)^{\otimes d}$ is a vertex operator algebra.
For simplicity we set  $G=\Z/2\Z\times \Z/2\Z$. Then any simple 
$(\mathscr{T}^+)^{\otimes d}$-module is obtained as 
\begin{equation*}
\mathscr{F}^{\vec{a}}=\bigotimes_{i=1}^d \mathscr{T}^{a_i}\quad\text{for some}\quad \vec{a}=(a_1,\dots,a_d)\in G^d.
\end{equation*}
Then, by Theorems \ref{thm-1-4} and \ref{thm-main-1}, we have
\begin{equation}\label{eq-fusion-product-tensor}
\mathscr{F}^{\vec{a}}\times\mathscr{F}^{\vec{b}}
=\mathscr{F}^{\vec{a}+\vec{b}}\quad(\vec{a},\vec{b}\in G^d)
\end{equation}
in the fusion algebra $\Fusion((\mathscr{T}^+)^{\otimes d})$.

Let $X(i)$ be the subset of $G^d$ consisting of elements $\vec{a}$ satisfying
\begin{equation*}
\sharp\{k\,|\,a_k^1=1\}\equiv i\mod2,
\end{equation*}
where we denote the $i$-th component of $\vec{a}\in G^{d}$ by
$a_i=(a_i^1,a_i^2)$.
Let $X(i,j)$ be the subset of $X(i)$ consisting of elements $\vec{a}$
with the property
\begin{equation*}
a_k^2=j
\end{equation*}
for all $1\le k\le d$.
Then the sets $X(i,j)\,(i,j\in\Z/2\Z)$ are mutually disjoint and 
\begin{equation}\label{eq-sets-sum}
X(i,j)+X(i^\prime,j^\prime)=X(i+i^\prime,j+j^\prime)\quad{for}\quad
i,i^\prime,j,j^\prime\in\Z/2\Z.
\end{equation}
We note that the set $\{X(i,j)\,|\,i,j\in\Z/2\Z\}$ is isomorphic to the group $G=\Z/2\Z\times \Z/2\Z$ as an abelian group.

For any subset $X\subseteq G^d$ we set
\begin{equation*}
\mathscr{F}^X=\bigoplus_{\vec{a}\in X}\mathscr{F}^{\vec{a}}=\bigoplus_{\vec{a}\in X}\bigotimes_{i=1}^d \mathscr{T}^{a_i}.
\end{equation*}
Since $(\mathscr{T}^+)^{\otimes d}$ is isomorphic to a subalgebra of 
$\mathscr{F}^+$ with the same Virasoro element,
every simple $\mathscr{F}^+$-module is a  simple $(\mathscr{T}^+)^{\otimes d}$-module.
In fact we have
\begin{equation}\label{eq-07202001}
\begin{cases}
&\mathscr{F}^+\cong\bigoplus_{\vec{a}\in X(0,0)}\mathscr{F}^a
=\mathscr{F}^{X(0,0)},\\
&\mathscr{F}^-\cong\bigoplus_{\vec{a}\in X(1,0)}\mathscr{F}^a
=\mathscr{F}^{X(1,0)},\\
&\mathscr{F}^+_t\cong\bigoplus_{\vec{a}\in X(0,1)}\mathscr{F}^a
=\mathscr{F}^{X(0,1)},\\
&\mathscr{F}^-_t\cong\bigoplus_{\vec{a}\in X(1,1)}\mathscr{F}^a
=\mathscr{F}^{X(1,1)}.
\end{cases}
\end{equation}

The same argument for $d=1$ given in the previous section shows that
fusion rules among triples
\begin{equation*}
\begin{split}
&(\mathscr{F}^+,\mathscr{F}^+,\mathscr{F}^+),\quad
(\mathscr{F}^+,\mathscr{F}^-,\mathscr{F}^-),\quad
(\mathscr{F}^+,\mathscr{F}^+_t,\mathscr{F}^+_t),\\
&(\mathscr{F}^+,\mathscr{F}^-_t,\mathscr{F}^-_t),\quad
(\mathscr{F}^-,\mathscr{F}^+_t,\mathscr{F}^-_t)
\end{split}
\end{equation*}
and their permutations are non-zero.
By \eqref{eq-sets-sum} and \eqref{eq-07202001} we have:

\begin{lemma}\label{lem-main-sub}
For any $X,\, Y,\,Z\in \{X(i,j)\;|\;i,j\in\Z/2\Z\,\}$ the fusion rule of type 
$\fusion{\mathscr{F}^X}{\mathscr{F}^Y}{\mathscr{F}^{X+Y}}$
is non-zero.
\end{lemma}

We are now in a position to state our main result.
\begin{theorem}
The fusion algebra $\Fusion(\mathscr{F}^+)$ for $d>1$ is isomorphic to 
$\Fusion(\mathscr{T}^+)$, i.e.,
\begin{equation*}
\dim_\C I_{\mathscr{F}^+}
\fusion{\mathscr{F}^X}{\mathscr{F}^Y}{\mathscr{F}^Z}
=\begin{cases}
1&\quad X+Y=Z,\\
0&\quad\text{otherwise},
\end{cases}
\end{equation*}
where $X,Y,Z\in\{X(i,j)\,|\,i,j\in\Z/2\Z\}$.
\end{theorem}
\begin{proof}
By the definition of the fusion product it follows that 
\begin{equation*}
\mathscr{F}^X\times\mathscr{F}^Y=\sum_{i,j=0,1}\
\dim_\C I_{\mathscr{F}^+}\fusion{\mathscr{F}^X}{\mathscr{F}^Y}{\mathscr{F}^{X(i,j)}}
\mathscr{F}^{X(i,j)}.
\end{equation*}
By Proposition \ref{prop-1-3}, for any $\vec{a}\in X$ and $\vec{b}\in Y$, 
\begin{equation*}
\begin{split}
\dim_\C I_{\mathscr{F}^+}
\fusion{\mathscr{F}^X}{\mathscr{F}^Y}{\mathscr{F}^{X(i,j)}}
&\le\dim_\C I_{\mathscr{F}^{\vec{0}}}
\fusion{\mathscr{F}^{\vec{a}}}{\mathscr{F}^{\vec{b}}}{\mathscr{F}^{X(i,j)}}\\
&=\sum_{\vec{c}\in X(i,j)}
\dim_\C I_{\mathscr{F}^{\vec{0}}}\fusion{\mathscr{F}^{\vec{a}}}{\mathscr{F}^{\vec{b}}}{\mathscr{F}^{\vec{c}}}.
\end{split}
\end{equation*}
By \eqref{eq-fusion-product-tensor} we see that 
\begin{equation*}
\dim_\C I_{\mathscr{F}^{\vec{0}}}
\fusion{\mathscr{F}^{\vec{a}}}{\mathscr{F}^{\vec{b}}}
{\mathscr{F}^{\vec{c}}}
=\begin{cases}
1&\quad \vec{a}+\vec{b}=\vec{c},\\
0&\quad\text{otherwise}.
\end{cases}
\end{equation*}
Hence by Lemma \ref{lem-main-sub} we have
\begin{equation*}
\dim_\C I_{\mathscr{F}^+}\fusion{\mathscr{F}^X}{\mathscr{F}^Y}{\mathscr{F}^{X(i,j)}}=
\begin{cases}
1&\quad X+Y=X(i,j),\\
0&\quad\text{otherwise}.
\end{cases}
\end{equation*}
\end{proof}

\begin{remark}
We can see that $\{\mathscr{F}^\pm,\mathscr{F}^\pm_t\}$ forms a group isomorphic to $G=\Z/2\Z\times\Z/2\Z$
and that fusion rules are independent of $d$.
\end{remark}

\appendix
\section{Appendix}\label{appendix}
In the appendix we will derive a formula of a determinant which is used in subsection  \ref{subsect-4-3}.
\subsection{A formula of a determinant}\label{appendix-1}

Let $k$ be a non-negative integer and let $a,\,b,\,c\,$ be formal variables. We define a matrix $A^{k}(a,b,c)$ by
\begin{equation*}
A^{k}_{p,q}(a,b,c)=
\begin{cases}
\displaystyle{\binom{q+a}{p}\binom{k-q+b}{k-p}}&\quad
(0\leq p\leq k,\, 0\leq q\leq k+1),\\
&\\
\displaystyle{\binom{q+c}{k+1}}
&\quad (p=k+1,\,0\leq q\leq k+1)
\end{cases}
\end{equation*}
and
\begin{equation*}
A^{k}(a,b,c)=(A^{k}_{j,i}(a,b,c))_{0\leq i,j\leq k+1}. 
\end{equation*}
We can obtain the exact form of the $\det A^{k}(a,b,c)$.
\begin{lemma}\label{theorem01}
For any  non-negative integer $k$ and formal variables $a,\,b,\,c$, we have 
\begin{equation*}
\det A^{k}(a,b,c)=\prod_{i=1}^{k}\binom{a+b+i}{i}.
\end{equation*}
\end{lemma}
\begin{proof}
Note that 
\begin{equation*}
\begin{split}
\sum_{p=0}^{k}A^k_{p,q}(a,b,c)&
=\sum_{p=0}^k\binom{q+a}{p}\binom{k-q+b}{k-p}\\
&=\left.\left(\sum_{p=0}^k\binom{q+a}{p}x^{q+a-p}\binom{k-q+b}{k-p}x^{p-q+b}\right)\right|_{x=1}\\
&=\left.\left(\sum_{p=0}^k(\partial^{(p)}x^{q+a})(\partial^{(k-p)}x^{k-q+b})\right)\right|_{x=1}\\
&=\left.\left(\partial^{(k)}x^{k+a+b}\right)\right|_{x=1}=\binom{k+a+b}{k}
\end{split}
\end{equation*}
for any $q$. For simplicity, we write 
$A^k_{p,q}(a,b,c)$ by $A^k_{p,q}$. 
Then we proceed
\begin{equation}\label{6.10}
\begin{split}
\det A^k(a,b,c)
&=
\begin{vmatrix}
A^k_{0,0}&\cdots &A^k_{k,0}&A^k_{k+1,0}\\
\vdots& &\vdots &\vdots \\
A^k_{0,q}&\cdots &A^k_{k,q }&A^k_{k+1,q}\\
\vdots&  &\vdots&\vdots\\
A^k_{0,k+1}&\cdots &A^k_{k,k+1}&A^k_{k+1,k+1}
\end{vmatrix}\\
&\\
&=
\begin{vmatrix}
A^k_{0,0}&\cdots& A^k_{k-1,0}&\sum_{p=0}^kA^k_{p,0}&A^k_{k+1,0}\\
\vdots & &\vdots &\vdots&\vdots\\
A^k_{0,q}&\cdots &A^k_{k-1,q}&\sum_{p=0}^{k}A^k_{p,q }&A^k_{k+1,q}\\
\vdots&  &\vdots&\vdots&\vdots\\
A^k_{0,k+1}&\cdots &A^k_{k-1,k+1}&\sum_{p=0}^{k}A^k_{p,k+1}&
A^k_{k+1,k+1}
\end{vmatrix}
\\
&\\
&=\begin{vmatrix}
A^k_{0,0}&\cdots& A^k_{k-1,0}&\binom{k+a+b}{k}&A^k_{k+1,0}\\
 \vdots  &&\vdots &\vdots &\vdots\\
A^k_{0,q}&\cdots& A^k_{k-1,q}&\binom{k+a+b}{k}&A^k_{k+1,q}\\
 \vdots&& \vdots &\vdots&\vdots\\
A^k_{0,k+1}&\cdots& A^k_{k-1,k+1}&\binom{k+a+b}{k}&A^k_{k+1,k+1}
\end{vmatrix}\\
&\\
&=\binom{k+a+b}{k}\begin{vmatrix}
A^k_{0,0}&\cdots& A^k_{k-1,0}&1&A^k_{k+1,0}\\
 \vdots & &\vdots &\vdots &\vdots\\
A^k_{0,q}&\cdots& A^k_{k-1,q}&1&A^k_{k+1,q}\\
 \vdots&&\vdots  &\vdots&\vdots\\
A^k_{0,k+1}&\cdots& A^k_{k-1,k+1}&1&A^k_{k+1,k+1}
\end{vmatrix}.
\end{split}
\end{equation}

On the one hand, for any $0\leq q\leq k$ and $1\leq p\leq k-1$, 
we see that 
\begin{equation*}
\begin{split}
B_{p,q}^{k}:&=A_{p,q}^{k}(a,b,c)-A_{p,q+1}^{k}(a,b,c)\\
&=\binom{q+a}{p}\binom{k-q+b}{k-p}-\binom{q+a+1}{p}
\binom{k-q+b-1}{k-p}\\
&\\
&=\binom{q+a}{p-1}\binom{k-q+b-1}{k-1-p}\frac{1}{p(k-p)}\\
&\qquad\qquad\times
\biggl((q+a-p+1)(k-q+b)-(q+a+1)(p-q+b)\biggr)\\
&\\
&=\binom{q+a}{p-1}\binom{k-q+b-1}{k-1-p}\left(\frac{q+a+1}{p}
-\frac{k-q+b}{k-p}\right)\\
&\\
&=\binom{q+a+1}{p}\binom{k-q+b-1}{k-1-p}-\binom{q+a}{p-1}
\binom{k-q+b}{k-p}\\
&\\
&=A^{k-1}_{p,q}(a+1,b,c)-A_{p-1,q}^{k-1}(a,b+1,c).
\end{split}
\end{equation*}
On the other hand, for $p=0$ and $p=k+1$, we have 
 \begin{align*}
 &A_{0,q}^{k}(a,b,c)-A_{0,q+1}^{k}(a,b,c)=\binom{k-1-q+b}{k-1}
 =A^{k-1}_{0,q}(a,b,c),\\
 &A_{k+1,q}^{k}(a,b,c)-A_{k+1,q+1}^{k}(a,b,c)=-\binom{q+c}{k}
 =-A^{k-1}_{k,q}(a,b,c). 
 \end{align*}
Therefore \eqref{6.10} turns to be
\begin{equation*}
\begin{split}
&\det A^{k}(a,b,c)\\
&=\binom{k+a+b}{k}
\begin{vmatrix}
A_{0,0}^{k-1}&B_{1,0}^{k}&\cdots &B_{k-1,0}^{k}&0&-A^{k-1}_{k,0}
\\
\vdots  & \vdots &&\vdots&\vdots&\vdots\\
A_{0,q}^{k-1} &B_{1,q}^{k}&\cdots &B^{k}_{k-1,q}&0&-A^{k-1}_{k,q}\\
\vdots &\vdots &  &\vdots&\vdots&\vdots\\
A_{0,k+1}^{k}&A^{k}_{1,k+1}&\cdots &A^{k}_{k-1,k+1}&1&A^k_{k+1,k+1}
\end{vmatrix}\\
&\\
&=\binom{k+a+b}{k}
\begin{vmatrix}
A_{0,0}^{k-1}&B_{1,0}^{k}&\cdots &B_{k-1,0}^{k}&A^{k-1}_{k,0}\\
\vdots  & \vdots &&\vdots&\vdots\\
A_{0,q}^{k-1} &B_{1,q}^{k}&\cdots &B^{k}_{k-1,q}&A^{k-1}_{k,q}\\
\vdots &\vdots &  &\vdots&\vdots\\
A_{0,k}^{k-1} &B_{1,k}^{k}&\cdots &B^{k}_{k-1,k}&
A^{k-1}_{k,k}
\end{vmatrix}.
\end{split}
\end{equation*}

Provided that
\begin{equation}\label{asdiuf}
A_{0,q}^{k-1}+\sum_{i=1}^{p}B_{i,q}^{k}=A_{p,q}^{k-1}.
\end{equation}
Then we have 
\begin{equation*}
\det A^{k}(a,b,c)=\binom{k+a+b}{k}\det A^{k-1}(a,b,c)
\end{equation*}
Since 
\begin{equation*}
\det A^{0}(a,b,c)=
\begin{vmatrix}
1&c\\
1&c+1
\end{vmatrix}=1,
\end{equation*}
we conclude that 
\begin{equation*}
\det A^{k}(a,b,c)=\prod_{i=1}^{k}\binom{i+a+b}{i}.
\end{equation*}

Thus it is enough  to show Identity \eqref{asdiuf}.  
We shall prove the identity by induction on $p$.
For $p=1$, we have
\begin{equation*}
\begin{split}
&A_{0,q}^{k-1}+B_{1,q}^{k}\\
&=A^{k-1}_{1,q}(a+1,b,c)-A_{0,q}^{k-1}(a,b+1,c)+A_{0,q}^{k-1}(a,b,c)\\
&=\binom{a+1+q}{1}\binom{k-1-q+b}{k-2}-\binom{k-q+b}{k-1}
+\binom{k-1-q+b}{k-1}\\
&\\
&=\binom{a+1+q}{1}\binom{k-1-q+b}{k-2}-\binom{k-q+b-1}{k-2}\\
&\\
&=\binom{a+q}{1}\binom{k-1-q+b}{k-2}=A_{1,q}^{k-1}(a,b,c).
\end{split}
\end{equation*}
Now suppose that 
\begin{equation*}
A_{0,q}^{k-1}+\sum_{i=1}^{p-1}B_{i,q}^{k}=A_{p-1,q}^{k-1}(a,b,c).
\end{equation*}
Then we have 
\begin{equation*}
\begin{split}
&A_{0,q}^{k-1}+\sum_{i=1}^{p}B_{i,q}^{k}\\
&=A_{p-1,q}^{k-1}(a,b,c)+A^{k-1}_{p,q}(a+1,b,c)
-A_{p-1,q}^{k-1}(a,b+1,c)\\
&\\
&=\binom{a+q}{p-1}\binom{k-1-q+b}{k-p}
+\binom{a+q+1}{p}\binom{k-1-q+b}{k-1-p}\\
&\quad-\binom{a+q}{p-1}\binom{k-q+b}{k-p}\\
&\\
&=-\binom{a+q}{p-1}\binom{k-1-q+b}{k-1-p}+\binom{a+q+1}{p}
\binom{k-1-q+b}{k-1-p}\\
&\\
&=\binom{a+q}{p}\binom{k-1-q+b}{k-1-p}=A_{p,q}^{k-1}(a,b,c).
\end{split}
\end{equation*}
This completes the proof of Identity \eqref{asdiuf}.
\end{proof}

\end{document}